\documentclass[11pt]{amsart}
\usepackage{geometry}
\usepackage{float}
\usepackage{caption}
\usepackage{subcaption}
\usepackage[round]{natbib}

\usepackage{amsthm, amsmath, amssymb}
\usepackage{tikz}
\usepackage{hyperref}
\usepackage{bm}
\usepackage{mathrsfs}
\numberwithin{equation}{section}
\newtheorem{thm}{Theorem}[section]
\newtheorem{defn}[thm]{Definition}
\newtheorem{lem}[thm]{Lemma}
\newtheorem{prop}[thm]{Proposition}
\newtheorem{cor}[thm]{Corollary}
\theoremstyle{remark}
\newtheorem{rmk}[thm]{Remark}
\newtheorem{claim}[thm]{Claim}

\DeclareMathOperator{\Bd}{Bound}
\DeclareMathOperator{\Gr}{Gr}
\newcommand{\Prj}{\text{Project}}
\newcommand{\mf}{\mathbf}
\newcommand{\reals}{\mathbb{R}}
\newcommand{\naturals}{\mathbb{N}}
\newcommand{\complexes}{\mathbb{C}}
\newcommand{\integers}{\mathbb{Z}}

\newcommand{\bdot}[2]{
\draw [black, fill = black] (#1, #2) circle [radius = 0.1];
}

\newcommand{\wdot}[2]{
\draw [black, fill = white] (#1, #2) circle [radius = 0.1];
}

\newcommand{\sdot}[2]{
\draw [black, fill=black] (#1, #2) circle [radius = 0.05]
}

\newcommand{\edge}[3]{
\draw (#1,#2) -- (#1, #2 + #3);
\bdot{#1}{#2}
\wdot{#1}{#2 + #3}
}

\newcommand{\crsg}[2]{
\begin{scope}[xshift = #1 cm, yshift = #2 cm]
\draw [rounded corners] (0,0) -- (0.2,0) -- (0.8,1)-- (1,1);
\draw [rounded corners] (0,1) -- (0.2,1) -- (0.8,0) -- (1,0);\
\end{scope}
}

\newcommand{\btm}[2]{
\begin{scope}[xshift = #1 cm, yshift = #2 cm]
\draw (0,0) -- (1,0);
\end{scope}
}

\newcommand{\tp}[2]{
\begin{scope}[xshift = #1 cm, yshift = #2 cm]
\draw (0,1) -- (1,1);
\end{scope}
}

\newcommand{\brd}[2]{
\begin{scope}[xshift = #1 cm, yshift = #2 cm]
\draw[dashed] (0.2,0) -- (0.8,1);
\draw[dashed] (0.2,1) -- (0.8,0);
\draw (0,0) -- (1,0);
\draw (0,1) -- (1,1);
\end{scope}
}

\newcommand{\mbrd}[2]{
\begin{scope}[xshift = #1 cm, yshift = #2 cm, scale = 0.5]
\draw[dashed] (0.2,0) -- (0.8,1);
\draw[dashed] (0.2,1) -- (0.8,0);
\draw (0,0) -- (1,0);
\draw (0,1) -- (1,1);
\end{scope}
}

\newcommand{\dbrd}[2]{
\begin{scope}[xshift = #1 cm, yshift = #2 cm, scale = 0.5]
\draw[dashed] (0.2,0) -- (0.8,1);
\draw[dashed] (0.2,1) -- (0.8,0);
\end{scope}
}

\newcommand{\dmr}[2]{
\begin{scope}[xshift = #1 cm, yshift = #2 cm]
\draw (0,0) -- (1,0);
\draw (0,1) -- (1,1);
\edge{0.5}{0}{1};
\end{scope}
}

\newcommand{\hdmr}[2]{
\begin{scope}[xshift = #1 cm, yshift = #2 cm]
\draw (0.5,0) -- (1,0);
\draw (0.5,1) -- (1,1);
\edge{0.5}{0}{1};
\end{scope}
}

\newcommand{\tdmr}[2]{
\begin{scope}[xshift = #1 cm, yshift = #2 cm]
\draw (0.5,0) -- (1,0);
\draw (0,1) -- (1,1);
\edge{0.5}{0}{1};
\end{scope}
}

\newcommand{\bdmr}[2]{
\begin{scope}[xshift = #1 cm, yshift = #2 cm]
\draw (0,0) -- (1,0);
\draw (0.5,1) -- (1,1);
\edge{0.5}{0}{1};
\end{scope}
}

\newcommand{\tcrsg}[2]{
\begin{scope}[xshift = #1 cm, yshift = #2 cm]
\draw [rounded corners] (0,1) -- (0.2,1) -- (0.8,0) -- (1,0);
\end{scope}
}

\newcommand{\bcrsg}[2]{
\begin{scope}[xshift = #1 cm, yshift = #2 cm]
\draw [rounded corners] (0,0) -- (0.2,0) -- (0.8,1)-- (1,1);
\end{scope}
}

\author{Rachel Karpman}
\address{
Department of Mathematics,
University of Michigan, Ann Arbor,
530 Church St.,
Ann Arbor, MI 48109-1043
USA}
\email{rkarpman@umich.edu}
\title[Bridge Graphs and Deodhar Parametrizations]{Bridge Graphs and Deodhar Parametrizations for Positroid Varieties}

\begin{document}

\begin{abstract}
A \emph{parametrization} of a positroid variety $\Pi$ of dimension $d$ is a regular map \begin{math}(\complexes^{\times})^{d} \rightarrow \Pi\end{math} which is birational onto a dense subset of $\Pi$.  There are several remarkable combinatorial constructions which yield parametrizations of positroid varieties.  We investigate the relationship between two families of such parametrizations, and prove they are essentially the same.  Our first family is defined in terms of Postnikov's \emph{boundary measurement map}, and the domain of each parametrization is the space of edge weights of a planar network.  We focus on a special class of planar networks called \emph{bridge graphs}, which have applications to particle physics.  
Our second family arises from Marsh and Rietsch's parametrizations of Deodhar components of the flag variety, which are indexed by certain subexpressions of reduced words.  Projecting to the Grassmannian gives a family of parametrizations for each positroid variety.  
We show that each Deodhar parametrization for a positroid variety corresponds to a bridge graph, while each parametrization from a bridge graph agrees with some projected Deodhar parametrization. 
\end{abstract}
\thanks{This research was supported by NSF grant number DGE 1256260.}

\maketitle{}
\tableofcontents{}

\section{Introduction}
\label{intro}

Lusztig defined the \emph{totally nonnegative part} of an abstract flag manifold $G/P$ and conjectured that it was made up of topological cells, a conjecture proved by Rietsch in the late 1990's \citep{Lus94, Lus98, Rie99}.  More than a decade later, Postnikov introduced the \emph{positroid stratification} of the totally nonnegative Grassmannian $\Gr_{\geq 0}(k,n)$, and showed that this stratification was a special case of Lusztig's \citep{Pos06}.  While Lusztig's approach relied on the machinery of canonical bases, Postnikov's was more elementary.  Each \emph{positroid cell} in Postnikov's stratification was defined as the locus in $\Gr_{\geq 0}(k,n)$ where certain Pl\"{u}cker coordinates vanish.  

The positroid stratification of $\Gr_{\geq 0}(k,n)$ extends to a stratification of the complex Grassmannian $\Gr(k,n)$ of $k$-planes in $n$-space.   That is, we can decompose $\Gr(k,n)$ into \emph{positroid varieties} $\Pi$ which are the Zariski closures of Postnikov's totally nonnegative cells.  Remarkably, these positroid varieties are the images of \emph{Richardson varieties} in $\mathcal{F}\ell(n)$ under the natural projection 
\begin{displaymath} \pi_k:\mathcal{F}\ell(n) \rightarrow \Gr(k,n).\end{displaymath}
Moreover, for each positroid variety $\Pi$, there is a family of Richardson varieties which project birationally to $\Pi$.

The stratification of $\Gr(k,n)$ by projected Richardson varieties was first studied by Lusztig \citep{Lus98}.  Brown, Goodearl and Yakimov investigated the same stratification from the viewpoint of Poisson geometry \citep{BGY06}.  Finally, Knutson, Lam and Speyer showed that Lusztig's strata were in fact the Zariski closures of Postnikov's totally nonnegative cells \citep{KLS13}.  

Postnikov defined a family of maps onto each positroid cell in $\Gr_{\geq 0}(k,n)$.  The domain of each map is the space of positive real edge weights of some weighted planar network, and there is a class of such networks for each positroid cell \citep{Pos06}.  Let $G$ be a planar network corresponding to a positroid cell 
\begin{math}\mathring{\Pi}_{\geq 0}\end{math} 
of dimension $d$ in $\Gr_{\geq 0}(k,n)$.  Specializing all but an appropriately chosen set of $d$ edge weights to $1$ yields a homeomorphism 
\begin{math}(\reals^+)^d \rightarrow \mathring{\Pi}_{\geq 0}\end{math} 
which we call a \emph{parametrization} of 
\begin{math}\mathring{\Pi}_{\geq 0}.\end{math} If we let the edge weights range over $\complexes^{\times}$ instead of $\reals^+$ we obtain a well-defined homeomorphism onto a dense subset of the positroid variety $\Pi$ in $\Gr(k,n)$ corresponding to the totally nonnegative cell 
\begin{math}\mathring{\Pi}_{\geq 0}\end{math}
\citep{MS14}.  We call these maps \emph{parametrizations} also.

In this paper, we investigate a particular class of network parametrizations, which arise from \emph{bridge graphs}.  Bridge graphs are constructed by an inductive process, and the definition of the corresponding parametrization is particularly straightforward.  In addition, bridge graphs have proven to be useful tool in particle physics \citep{ABCGPT14}.

Another method for parametrizing positroid varieties arises from the Deodhar decompositions of the flag variety $\mathcal{F}\ell(n)$, defined by Deodhar in \citep{Deo85}.  Each Deodhar decomposition of $\mathcal{F}\ell(n)$ refines the \emph{Richardson decomposition}.  Richardson varieties are indexed by pairs of permutations $u,w$, where $u \leq w$ in the Bruhat order on the symmetric group $S_n$, or equivalently by intervals $[u,w]$ in Bruhat order.  To define a Deodhar decomposition of $\mathcal{F}\ell(n)$, we first choose a reduced word $\mf{w}$ for each element $w$ of $S_n$.   For each of our chosen reduced words $\mf{w}$, and each $u \in S_n$ with $u \leq w$ in Bruhat order, we can express the (open) Richardson variety \begin{math} \mathring{X}_u^w\end{math} indexed by $[u,w]$ as a disjoint union of subsets called \emph{Deodhar components}.  The Deodhar components of  \begin{math} \mathring{X}_u^w\end{math} are indexed by \emph{distinguished subexpressions} for $u$ in $\mf{w}$; that is, by subwords for $u$ in $\mf{w}$ which satisfy a technical condition.
 
Let 
\begin{math}\Pi \subset \Gr(k,n)
\end{math}
be a positroid variety, and fix a Deodhar decomposition of $\mathcal{F}\ell(n)$.  We have a family of Deodhar components 
\begin{math} \mathcal{D} \subset \mathcal{F}\ell(n)\end{math}
such that the natural projection from $\mathcal{F}\ell(n)$ to $\Gr(k,n)$ maps each $\mathcal{D}$ isomorphically to a dense subset of $\Pi$.   These are precisely the top-dimensional Deodhar components of the Richardson varieties $\mathring{X}_u^w$ which project birationally to $\Pi$.  For each $\mathring{X}_u^w$, the desired component is indexed by a special choice of subexpression for $u$ in $\mf{w}$ called a \emph{positive distinguished subexpression}, or PDS.  Marsh and Rietsch defined explicit matrix parametrizations for each Deodhar component of $\mathcal{F}\ell(n)$ \citep{MR04}. Composing these parametrizations with the projection to $\Gr(k,n)$, we have a family of explicit parametrizations for the positroid variety $\Pi$, which we call \emph{projected Deodhar parametrizations} or simply \emph{Deodhar parametrizations} \citep{TW13}.  

We will show that these two ways of parametrizing positroid varieties--via bridge graphs, and via projected Deodhar parametrizations--are essentially the same.  This relationship was first conjectured by Thomas Lam \citep{Lam13B}.  Our main result is the following.

\begin{thm} \label{main} Let $\Pi$ be a positroid variety in $\Gr(k,n)$.  For each Deodhar parametrization of $\Pi$, there is a bridge graph which yields the same parametrization.  
Conversely, any bridge graph parametrization of $\Pi$ agrees with some Deodhar parametrization.
\end{thm}

\begin{figure}[ht]
\centering
\begin{tikzpicture}[scale = 0.7]
\draw (4,0) -- (4,5);
\draw (4,1) -- (0,1);
\draw (4,2) -- (2,2);
\draw (4,3) -- (1,3);
\draw (4,4) -- (0,4);
\edge{0}{1}{3}; \node[left] at (0,2.5) {$t_1$};
\edge{1}{1}{2}; \node[left] at (1,2) {$t_2$};
\edge{2}{2}{1}; \node[left] at (2,2.5) {$t_3$};
\edge{3}{3}{1}; \node[left] at (3,3.5) {$t_4$};
\wdot{0.5}{1};
\wdot{4}{1};
\wdot{4}{2};
\wdot{4}{3};
\bdot{4}{4};
\bdot{1.5}{4};
\bdot{1.5}{3};
\foreach \x in {1,...,4}
	{\pgfmathtruncatemacro{\label}{5 - \x}
	 \node [right] at (4.1, \x) {\label};} 
\end{tikzpicture}
\caption{A bridge network.  All unlabeled edges have weight $1$.}
\label{firstfig}
\end{figure}
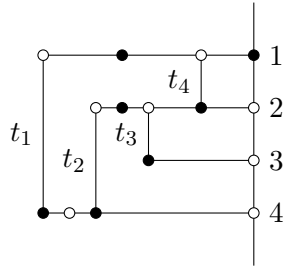
	
To convey the flavor of this result, we briefly sketch an example; the details will appear later.  Take $k=2$ and $n=4$.  Let $u = 2134$ and $w=4321$.  The Richardson variety $\mathring{X}_u^w$ projects birationally to the positroid \begin{math}\Pi_{\langle u,w \rangle_2}\end{math}.  The bridge graph in figure \ref{firstfig} yields a parametrization for this positroid variety, given by
\begin{equation}(t_1,t_2,t_3,t_4) \mapsto  \begin{bmatrix}
1 & t_4 & 0 & -t_1\\
0 & 1 & t_3 & t_2
\end{bmatrix}
\end{equation}

We claim that we can obtain the same map from some Deodhar parametrization.  Indeed, fix the reduced word 
\begin{math}\mf{w} = s_1s_2s_3s_2s_1s_2\end{math}
for $w$. The positive distinguished subexpression $\mf{u}$ for $u$ in $\mf{w}$ comprises the $s_3$ in position $3$ from the left, and the $s_1$ in position $5$, so we have the projected Deodhar parametrization
\begin{equation}(t_1,t_2,t_3,t_4) \mapsto \begin{bmatrix}
1 & 0 & 0 & 0\\
0 & 1 & 0 & 0
\end{bmatrix}
x_2(t_1)\dot{s_1}^{-1}x_2(t_2)\dot{s_3}^{-1}x_2(t_3)x_1(t_4)=\begin{bmatrix}
0 & 1 & t_3 & t_2\\
-1 & -t_4 & 0 & t_1
\end{bmatrix}\end{equation}
Note that 
\begin{equation}\begin{bmatrix}
0 & 1 & t_3 & t_4\\
-1 & -t_4 & 0 & t_1
\end{bmatrix}=
\begin{bmatrix} 0 & 1 \\
-1 & 0\\
\end{bmatrix}
\begin{bmatrix}
1 & t_4 & 0 & -t_1\\
0 & 1 & t_3 & t_2
\end{bmatrix}
\end{equation}

Hence, the two parametrizations send the point 
\begin{math} (t_1,t_2,t_3,t_4)\end{math}
to matrices which have the same row space, and hence represent the same point in the $\Gr(2,4)$

This work builds on a number of earlier results.  Postnikov's Le-diagrams, which index positroid varieties, provided one of the earliest links between planar networks and PDS's. A Le diagram is a Young diagram filled with $0$'s and $+$'s according to certain rules.  There is an beautiful bijection between Le-diagrams and PDS's of \emph{Grassmannian} permutations, permutations with a single descent at position $k$.  Moreover, Postnikov constructed a planar network from each Le-diagram, which yields a parametrization of the corresponding positroid variety \citep{Pos06}.  

Talaska and Williams explored the link between distinguished subexpressions and network parametrizations further in \citep{TW13}.  They considered Deodhar components of $\mathcal{F}\ell(n)$ indexed by all distinguished subexpressions of Grassmannian permutations, not just PDS's.  These Deodhar components project isomorphically to subsets of $\Gr(k,n)$, and the projections give a decomposition of $\Gr(k,n)$.  Marsh and Rietsch's work gave a unique parametrization for each component \citep{MR04}.  Talaska and Williams proved that each of these parametrizations arises from a network, which they constructed explicitly. For components indexed by PDS's, they recovered the planar networks corresponding to Postnikov's Le diagrams; for the remaining components, their networks were not all planar.  

In this paper, we relax the requirement that our Deodhar components be indexed by subexpressions of Grassmannian permutations.  Instead, we restrict our attention to Deodhar components which correspond to PDS's, and which project isomorphically to a subset of $\Gr(k,n)$.   These are exactly the Deodhar components which map to  dense subsets of positroid varieties.  We have a family of such components for each positroid variety, which in turn gives a family of parametrizations.  We show that each of these parametrizations arises from a planar network, which we construct explicitly.  The Le-diagrams defined by Postnikov, and recovered by Talaska and Williams, are a special case of this result.

The organization of the paper is as follows.  In Section \ref{back}, we provide the necessary background on combinatorial objects that index positroid varieties, the positroid decomposition of the Grassmannian, bridge graphs, distinguished subexpressions, and parametrizations of Deodhar components. Subsection \ref{diagrams} introduces \emph{bridge diagrams}, which are distinct from bridge graphs, and which play an extensive role in the proof of Theorem \ref{main}.  In Section \ref{mainproof}, we prove that every projected Deodhar parametrization agrees with some bridge graph parametrization, and in Section \ref{converse} we prove the converse.  The correspondence between Deodhar parametrizations and bridge graphs is not a bijection; rather, we have a family of Deodhar parametrizations for each bridge graph.  In Section \ref{isotopy} we define an equivalence relation among Deodhar parametrizations, such that each equivalence class corresponds to a unique bridge graph.

\subsection*{Acknowledgements} I would like to thank Thomas Lam for introducing me to his conjecture about the relationship between bridge graphs and Deodhar parametrizations, and for many helpful conversations.  My thanks also to Greg Muller and David E Speyer, for sharing their manuscript \citep{MS14}.  Finally, I am grateful  to Greg Muller, David E Speyer and Timothy M. Olson for productive discussions, and to Jake Levinson for useful suggestions in preparing this manuscript.

\section{Background}
\label{back}

\subsection{Notation for partitions and permutations}

Let $S_n$ denote the symmetric group in $n$ letters, with simple generators
\begin{math} s_1,\ldots,s_{n-1},\end{math}
and let $\ell$ denote the standard length function on $S_n$.  Write $[a]$ for the set 
\begin{math}\{1,2,\ldots,a\} \subseteq \naturals,\end{math}
and let $[a,b]$ denote the set 
\begin{math}\{a,a+1,\ldots,b\}\end{math}
for $a \leq b$.  For $a > b$, we set 
\begin{math}[a,b] = \emptyset\end{math}.
 For $k \leq n$, write 
 \begin{math} S_k \times S_{n-k}\end{math}
  for the Young subgroup of $S_n$ which fixes the sets $[k]$ and $[k+1,n]$.  A permutation is \emph{Grassmannian} of type $(k,n)$ if it is of minimal length in its left coset of 
  \begin{math} S_k \times S_{n-k},\end{math} and \emph{anti-Grassmannian} if it is of maximal length.  Alternatively, a permutation is Grassmannian if it is increasing on the sets $[k]$ and $[k+1,n]$, and anti-Grassmannian if it is decreasing on $[k]$ and $[k+1,n]$.

For \begin{math} u,w \in S_n\end{math}, we write $u \leq w$ to denote a relation in the (strong) Bruhat order. A factorization 
\begin{math} u=vw \in S_n\end{math}
 is \emph{length additive} if 
\begin{equation} \ell(u) = \ell(v)+\ell(w). \end{equation}
We use 
\begin{math} u \leq_{(r)} w\end{math}
 to denote a relation in the \emph{right weak order}, so 
\begin{math} u \leq_{(r)} w\end{math}
if there exists $v \in S_n$ such that $uv= w$ and the factorization is length additive.  Similarly, we write 
\begin{math} u \leq_{(l)} w\end{math} 
and to denote \emph{left weak order}, and say 
\begin{math} u \leq_{(l)} w\end{math} if there is a length-additive factorization $vu = w$.  All functions and permutations act on the left, so 
\begin{math} \sigma \rho \end{math}
means ``first apply $\rho$, then apply $\sigma$ to the result.''

 For $w \in S_n$, $w([a])$ denotes the unordered set 
 \begin{math} \{w(1),w(2),\ldots,w(a)\}.\end{math}
 Let 
 \begin{math} I,J \subseteq \naturals\end{math} with 
 
 \begin{align}
 I &= \{i_1 < i_2 < \ldots < i_m\}\\
J &= \{j_1 < j_2 < \ldots < j_m\}
 \end{align}
 
 We say $I \leq J$ if $i_r \leq j_r$ for all $1 \leq r \leq m$.  We denote the set of all $k$-element subsets of $[n]$ by ${{[n]}\choose k}$.

\subsection{Bruhat intervals and bounded affine permutations}

Fix $k \leq n$.  The $k$-Bruhat order on $S_n$, introduced by Bergeron and Sottile in \citep{BS98}, is defined as follows.  For 
\begin{math} u,w \in S_n,\end{math}
we say that $w$ is a $k$-cover of $u$, written 
\begin{math} w \gtrdot_k u,\end{math} if 
\begin{math} w \gtrdot u\end{math}
in Bruhat order and 
\begin{math} w([k]) \neq u([k])\end{math}.    To obtain the $k$-Bruhat order on $S_n$, we take the transitive closure of these cover relations, which remain cover relations. We use the symbol $\leq_k$ to denote $k$-Bruhat order.  If 
\begin{math} u \leq_k w,\end{math}
we write $[u,w]_k$ for the $k$-Bruhat interval 
\begin{math}\{v \mid u \leq_k v \leq_k w\}.\end{math}  
We make extensive use of the following criterion for comparison in $k$-Bruhat order.
 
\begin{thm}\label{kcomp}{\em(\citep{BS98}, Theorem A)}
Let 
\begin{math}u, w \in S_n.\end{math}  Then 
\begin{math}u \leq_k w\end{math}
 if and only if
\begin{enumerate}
\item 
\begin{math}1 \leq a \leq k < b \leq n\end{math}
 implies 
 \begin{math}u(a) \leq w(a)\end{math}
 and 
 \begin{math}u(b) \geq w(b)\end{math}.
\item If $a < b$, 
\begin{math}u(a) < u(b)\end{math}
 and
 \begin{math}w(a) > w(b),\end{math}
 then 
 \begin{math} a \leq k < b.\end{math}
\end{enumerate}
\end{thm}

Following \citep[Section 2.3]{KLS13}, we define an equivalence relation on $k$-Bruhat intervals, generated by setting 
\begin{math}[u,w]_k \sim [x,y]_k\end{math}
 if there is some 
 \begin{math} z \in S_k \times S_{n-k}\end{math}
 such that $x = uz$ and $y=wz$, with both factorizations length additive.  We write 
\begin{math} \langle u, w \rangle_k\end{math}
for the equivalence class of $[u,w]_k$, and denote the set of all such classes by $\mathcal{Q}(k,n)$.  There is a partial order on $\mathcal{Q}(k,n),$ defined by setting 
\begin{math}\langle u, w \rangle_k \leq \langle x, y \rangle_k\end{math} 
if there exist representatives $[u',w']_k$ of 
\begin{math} \langle u,w \rangle_k\end{math}
 and $[x',y']_k$ of 
 \begin{math} \langle x, y \rangle_k\end{math}
 with \begin{math} [x',y']_k \subseteq [u',w']_k.\end{math}

The poset $\mathcal{Q}(k,n)$ was first studied by Rietsch, in the context of closure relations for totally nonnegative cells in general flag manifolds \citep{Rie06}.  Williams proved a number of combinatorial results about this poset, also in a more general setting \citep{Wil07}.  Our notation and conventions for the Grassmannian case are from \citep[Section 2]{KLS13}, where the authors discuss the combinatorics of $\mathcal{Q}(k,n)$ in some depth.  

We now recall some facts from \citep[Section 3]{KLS13} about \emph{bounded affine permutations}.  The poset of bounded affine permutations is isomorphic to $\mathcal{Q}(k,n)$, and anti-isomorphic to the poset of Postnikov's \emph{decorated permutations} \citep{Pos06, KLS13}.

\begin{defn} An \textbf{affine permutation} of order $n$ is a bijection 
\begin{math}f:\integers \rightarrow \integers\end{math}
which satisfies the condition 
\begin{equation} \label{affine} f(i+n) = f(i)+n\end{equation}
 for all 
 \begin{math} i \in \integers.\end{math}
 The affine permutations of order $n$ form a group, which we denote 
 \begin{math} \widetilde{S}_n.\end{math}
 \end{defn}

We may embed $S_n$ in $\widetilde{S}_n$ by extending each permutation periodically, in accordance with \eqref{affine}. 

\begin{defn}
An affine permutation of order $n$ has type $(k,n)$ if
\begin{equation}\displaystyle \frac{1}{n} \sum_{i=1}^n (f(i)-i)=k.\end{equation}
We denote the set of affine permutations of type $(k,n)$ by 
\begin{math}\widetilde{S}^k_n.\end{math}
\end{defn}

Affine permutations of type $(0,n)$ form an infinite Coxeter group, the \emph{affine symmetric group}, often denoted 
\begin{math}\widetilde{A}_{n-1}.\end{math}  This group has simple generators 
\begin{math}s_1,\ldots,s_{n},\end{math}
where $s_i$ is the affine permutation which interchanges $i+rn$ and $i+1+rn$, for each 
\begin{math} r \in \integers.\end{math}
  The reflections in 
  \begin{math} \widetilde{S}^0_n\end{math}
  are given by $t_{(i,j)}$ for 
  \begin{math} i \not\equiv j \pmod{n},\end{math}
  where $t_{(i,j)}$ is the affine permutation which interchanges $i+rn$ and $j+rn$ for all 
  \begin{math}r \in \integers.\end{math}

The Bruhat order on 
\begin{math}\widetilde{S}^0_n
\end{math}
is defined by taking the transitive closure of the relations 
\begin{math} u \rightarrow w\end{math}, 
where we set
\begin{math} u \rightarrow w\end{math} 
if there exists 
\begin{math} i<j \in \integers\end{math}
with 
\begin{math} i \not\equiv j \pmod{n},\end{math} 
such that 
\begin{math}u(i) < u(j)\end{math} 
 and 
 \begin{math} v = ut_{(i,j)}.\end{math}  For further discussion, see \citep[Section 2.1]{BB05} and \citep[Section 8.3]{BB05}

There is a natural bijection 
\begin{math} \widetilde{S}^k_n \rightarrow \widetilde{S}^0_n,\end{math} which takes 
\begin{math} f \in \widetilde{S}^k_n\end{math} 
to the function 
\begin{math} i \mapsto f(i-k)\end{math}.
Hence the Bruhat order on 
\begin{math}\widetilde{S}^0_n\end{math} induces a partial order $\leq$ on 
\begin{math} \widetilde{S}^k_n.\end{math}
Abusing terminology, we call this the \emph{Bruhat order} on 
\begin{math}\widetilde{S}^k_n.\end{math}  
  We denote the Bruhat order on 
  \begin{math} \widetilde{S}^0_n\end{math}
   by $\leq$.

\begin{defn}
An affine permutation in 
\begin{math} \widetilde{S}_n\end{math}
 is \textbf{bounded} if it satisfies the condition
\begin{equation}i \leq f(i) \leq f(i+n) \text{ for all }i \in \integers.\end{equation}
We write $\Bd(k,n)$ for the set of all bounded affine permutations of type $(k,n)$.  
\end{defn}

The set $\Bd(k,n)$ inherits the Bruhat order from 
\begin{math} \widetilde{S}^k_n.\end{math}  In fact, $\Bd(k,n)$ is a lower order ideal in 
\begin{math} \widetilde{S}^k_n\end{math}  
\cite[Lemma 3.6]{KLS13}
Similarly, the length function $\ell$ on 
\begin{math} \widetilde{S}_n^0\end{math}
 induces a grading on $\Bd(k,n)$, which we again denote $\ell$.  To find the rank of 
\begin{math} f \in \Bd(k,n),\end{math}
we count equivalence classes of inversions of $f$, defined below.
An \emph{inversion} of a bounded affine permutation $f$ is a pair $i < j$ such that 
\begin{math} f(i) > f(j).\end{math}
Two inversions $(i,j)$ and $(i',j')$ are equivalent if 
\begin{math} i'=i+rn\end{math}
and
\begin{math}j'=j+rn\end{math}
for some
\begin{math} r \in \integers.\end{math} 
The length of $f$ is the number of equivalence classes of inversions \citep[Theorem 5.9]{KLS13}.

For 
\begin{math} J \in {{n}\choose{k}},\end{math}
 we define the \emph{translation element} 
 \begin{math} t_J \in \widetilde{S}_n\end{math} by setting 
\begin{equation}t_J(i)=\begin{cases}
i+n & i \in J\\
i & i \not\in J
\end{cases}\end{equation}
for 
\begin{math}1 \leq i \leq n,\end{math}
 and extending periodically.
Every 
\begin{math} f \in \Bd(k,n)\end{math} may be written in the form
\begin{equation}f=\sigma t_{\mu}=t_{\nu}\sigma\end{equation}
where 
\begin{math} \sigma \in S_n,\end{math}
with 
\begin{math}\mu, \nu \in {{n}\choose{k}},\end{math} 
and $t_{\mu}$, $t_{\nu}$ translation elements.  Both factorizations are unique.

We now give an isomorphism between $\mathcal{Q}(k,n)$ and $\Bd(k,n)$.  Fix 
\begin{math}\langle u,w \rangle_k \in \mathcal{Q}(k,n)\end{math}.  
Let $\omega_k$ be the element of $\integers^n$ whose first $k$ entires are $1$'s, and whose remaining entries are $0$'s.  The function 
\begin{equation}\label{Q2Bd} f_{u,w} = ut_{\omega_k}w^{-1} \end{equation}
is a bounded affine permutation of type $(k,n)$.  If $[u',w']_k$ is any other representative of 
\begin{math} \langle u,w \rangle_k,\end{math}
then
\begin{math} f_{u',w'}=f_{u,w}.\end{math}
Hence this process gives a well-defined map 
\begin{math} \mathcal{Q}(k,n) \rightarrow \Bd(k,n),\end{math} which is in fact an isomorphism of posets \citep[Section 3.4]{KLS13}.
Note that for the translation elements $t_{w([k])}$ and $t_{u([k])},$ we have
\begin{equation} \label{permfactors} f_{u,w} = uw^{-1}t_{w([k])}=t_{u([k])}uw^{-1} \end{equation}
This follows easily from \eqref{Q2Bd}.

\subsection{Grassmannians, flag varieties, and Richardson varieties}

Let $\Gr(k,n)$ denote the Grassmannian of $k$-dimensional linear subspaces of the vector space $\complexes^n$.  We may realize $\Gr(k,n)$ as the space of full-rank $k \times n$ matrices modulo the left action of $\text{GL}(k)$, the group of invertible $k\times k$ matrices; a matrix $M$ represents the space spanned by its rows.  We number the rows of our matrices from top to bottom, and the columns from left to right.

The \emph{Pl{\"u}cker embedding}, which we denote $p$, maps $\Gr(k,n)$ into the projective space 
\begin{math}\mathbb{P}^{ {n\choose k}-1}\end{math}
with homogeneous coordinates $x_J$ indexed by the elements of 
\begin{math}{{[n]}\choose k}.\end{math}
For 
\begin{math} J \in {[n] \choose k}\end{math}
let $\Delta_J$ denote the minor with columns indexed by $J$.  Let $V$  be an $n$-dimensional subspace of $\complexes^n$ with representative matrix $M$.  Then $p(V)$ is the point defined by 
\begin{math} x_J = \Delta_J(M)\end{math}.
This map embeds $\Gr(k,n)$ as a smooth projective variety in 
\begin{math} \mathbb{P}^{ {n\choose k}-1}.\end{math}
The homogeneous coordinates $\Delta_J$ are known as \emph{Pl{\"u}cker  coordinates} on $\Gr(k,n)$.  The \emph{totally nonnegative Grassmannian}, denoted $\Gr_{\geq 0}(k,n),$ is the subset of $\Gr(k,n)$ whose Pl{\"u}cker coordinates are all nonnegative real numbers, up to multiplication by a common scalar. 

The \emph{flag variety} $\mathcal{F}\ell(n)$ is an algebraic variety whose points correspond to flags
\begin{equation}V_{\bullet} = \{0 \subset V_1\subset V_2 \subset \cdots \subset V_n = \complexes^n\}\end{equation}
where $V_i$ is a subspace of $\complexes^n$ of dimension $i$.  
We view $\mathcal{F}\ell(n)$ as the quotient of $\text{GL}(n)$ by the left action of $B$, the group of $n \times n$ lower triangular matrices.  Hence an $n \times n$ matrix $M$ represents the flag whose $i^{th}$ subspace is the span of the first $i$ rows of $M$. Note that these conventions differ from those used in \citep{MR04}, so we will use slightly different conventions for our matrix parametrizations. 

There is a natural projection 
\begin{math} \pi_k:\mathcal{F}\ell(n) \rightarrow \Gr(k,n),\end{math}
which carries a flag $V_{\bullet}$ to the $k$-plane $V_k$.  
If $V$ is a representative matrix for $V_{\bullet}$, the first $k$ rows of $V$ give a representative matrix $M$ for $V_k$.  Let $I_k$ be the $k \times n$ matrix whose first $k$ columns form a copy of the identity matrix, and whose remaining entries are $0$.  Then we have $M = I_k V$.

We now recall the definitions of Schubert and Richardson varieties, following the conventions of  \citep[Section 4]{KLS13}.  For a subset $J$ of $[n]$, let 
\begin{math} \Prj_J:\complexes^n \rightarrow \complexes^J \end{math} be projection onto the coordinates indexed by $J$.  For a permutation $w \in S_n$, we define the \emph{Schubert cell} corresponding to $w$ by
\begin{equation}\mathring{X}_{w}=\{ V_{\bullet} \in \mathcal{F}\ell(n) \mid \dim(\Prj_{[j]}(V_i)) = |w([i]) \cap [j]| \text{ for all } i\}\end{equation}
The closure of $\mathring{X}_w$ is the the \emph{Schubert variety}
\begin{equation}X_{w}=\{ V_{\bullet} \in \mathcal{F}\ell(n) \mid \dim(\Prj_{[j]}(V_i)) \leq |w([i]) \cap [j]| \text{ for all } i\}\end{equation}

Similarly, we define the \emph{opposite Schubert cell} and \emph{opposite Schubert variety} respectively by 
\begin{equation}\mathring{X}^{w}=\{ V_{\bullet} \in \mathcal{F}\ell(n) \mid \dim(\Prj_{[n-j+1,n]}(V_i)) = |w([i]) \cap [n-j+1,n]| \text{ for all } i\}\end{equation}
\begin{equation}X^{w}=\{ V_{\bullet} \in \mathcal{F}\ell(n) \mid \dim(\Prj_{[n-j+1,n]}(V_i)) \leq |w([i]) \cap [n-j+1,n]| \text{ for all } i\}\end{equation}

The Schubert cells (respectively, opposite Schubert cells) form a stratification of $\mathcal{F}\ell(n)$, which has been studied extensively.
We define the \emph{open Richardson variety} 
\begin{math} \mathring{X}_u^w\end{math}
 to be the intersection 
 \begin{math} \mathring{X}_u \cap \mathring{X}^w.\end{math}  The closure of 
 \begin{math} \mathring{X}_u^w\end{math}
 is the \emph{Richardson variety} $X_u^w.$
The variety $X_u^w$ is empty unless $u \leq w$, in which case it has dimension 
\begin{math} \ell(w)-\ell(u).\end{math}
Open Richardson varieties form a stratification of $\mathcal{F}\ell(n)$, which refines the Schubert stratification.

\subsection{Positroid varieties}

Let 
\begin{math}V \in \Gr_{\geq 0}(k,n)\end{math}.
The indices of the non-vanishing Pl{\"u}cker  coordinates of $V$ give a set
\begin{math} \mathcal{J} \subseteq {{[n]}\choose k}\end{math}
called the \emph{matroid} of $V$.  We define the \emph{matroid cell} 
\begin{math} \mathcal{M}_{\mathcal{J}}\end{math}
as the locus of points $V \in \Gr_{\geq 0}(k,n)$ with matroid $\mathcal{J}$.  The nonempty matroid cells in $\Gr_{\geq 0}(k,n)$ are the \emph{positroid cells} defined by Postnikov.  Positroid cells form a stratification of $\Gr_{\geq 0}(k,n)$, and each cell is homeomorphic to 
\begin{math}(\reals^+)^d\end{math} for some $d$ \citep[Theorem 3.5]{Pos06}.

The positroid stratification of $\Gr_{\geq 0}(k,n)$ extends to the complex Grassmannian $\Gr(k,n)$.  Taking the Zariski closure of a positroid cell of $\Gr_{\geq 0}(k,n)$ in $\Gr(k,n)$ gives a \emph{positroid variety}.  For a positroid variety
\begin{math} \Pi \subseteq \Gr(k,n),\end{math} we define the \emph{open positroid variety} 
\begin{math} \mathring{\Pi} \subset \Pi\end{math}
by taking the complement in $\Pi$ of all lower-dimensional positroid varieties.  The open positroid varieties give a stratification of $\Gr(k,n)$ \citep{KLS13}.

Positroid varieties in $\Gr(k,n)$ may be defined in numerous other ways.  There is a beautiful description of positroid varieties as intersections of cyclically permuted Schubert varieties.  In particular, the positroid decomposition refines the well-known Schubert decomposition of $\Gr(k,n)$, and shares many of its nice geometric properties \citep{KLS13}.  

Remarkably, positroid varieties in $\Gr(k,n)$ coincide with projected Richard varieties \citep[Section 5.4]{KLS13}. Indeed, let \begin{math} u \leq_k w.\end{math}
The projection $\pi_k$ maps 
\begin{math} \mathring{X}_u^w\end{math}
homeomorphically onto its image, which is an open positroid variety 
\begin{math} \mathring{\Pi}_{u,w}.\end{math}  
The closure of $\mathring{\Pi}_{u,w}$ is a (closed) positroid variety $\Pi_{u,w}$ and we have
\begin{equation} \pi_k(X_u^w)=\Pi_{u,w}.\end{equation}
Moreover, if $[u',w']_k$ is any representative of 
\begin{math} \langle u, w \rangle_k \in \mathcal{Q}(k,n),\end{math}
then 
\begin{math} \mathring{\Pi}_{u,w} = \mathring{\Pi}_{u',w'}\end{math}  
which implies \begin{math} \Pi_{u,w} = \Pi_{u',w'}\end{math}.
Hence each class 
\begin{math}\langle u, w \rangle_k \in \mathcal{Q}(k,n)\end{math} corresponds to a unique open positroid variety 
\begin{math}\mathring{\Pi}_{\langle u,w \rangle_k} = \pi_k(\mathring{X}_u^w )\end{math}
with closure $\Pi_{\langle u, w \rangle_k}$.
This correspondence gives an isomorphism between $\mathcal{Q}(k,n)$ and the poset of positroid varieties, ordered by \emph{reverse} inclusion \citep[Section 5.4]{KLS13}.

Since $\mathcal{Q}(k,n)$ is isomorphic to the poset of positroid varieties, so is $\Bd(k,n)$.  We write $\Pi_f$ for the positroid variety corresponding to a bounded affine permutation $f$. The length of $f$ gives the \emph{codimension} of the corresponding positroid variety \citep{KLS13}.  There are numerous other combinatorial objects that index positroid varieties, including Le diagrams and Grassmann necklaces.  For details, see \citep{Pos06, KLS13, MS14}.  

\subsection{Plabic graphs and bridge graphs}

A \emph{plabic graph} is a planar graph embedded in a disk, with each vertex colored black or white.  (Plabic is short for planar bicolor.)  The boundary vertices are numbered $1,2,\ldots,n$ in clockwise order, and all boundary vertices have degree one.  We call the edges adjacent to boundary vertices \emph{legs} of the graph, and a leaf adjacent to a boundary vertex a \emph{lollipop}.

Postnikov introduced plabic graphs in \citep[Section 11.5]{Pos06},  where he used them to construct parametrizations of positroid cells in the totally nonnegative Grassmannian.  In this paper, we follow the conventions of \citep{Lam13a}, which are more restrictive than Postnikov's. In particular, we require our plabic graphs to be bipartite, with the black and white vertices forming the partite sets.  An \emph{almost perfect matching} on a plabic graph is a subset of its edges which uses each interior vertex exactly once; boundary vertices may or may not be used.  We consider only plabic graphs which admit an almost perfect matching.

We can write any plabic graph as a union of paths and cycles, as follows.  Start with any edge $e=\{u,v\}$.  Begin traversing this edge in either direction, say 
\begin{math} u \rightarrow v.\end{math}
Turn (maximally) left at every white vertex, and (maximally) right at every black vertex.  Repeating this process gives a description of $G$ as a union of directed paths and cycles, called \emph{trips}.  Each edge is used twice in this decomposition, once in each direction.  Given a plabic graph $G$ with $n$ boundary vertices, we define the \emph{trip permutation} 
\begin{math} \sigma_G \in S_n\end{math}
of $G$ by setting 
\begin{math} \sigma_G(a) = b\end{math}
if the trip that start at boundary vertex $a$ ends at boundary vertex $b$.  

A plabic graph $G$ is \emph{reduced} if it satisfies the following criteria
\begin{enumerate}
\item $G$ has no trips which are cycles.
\item $G$ has no leaves, except perhaps some which are adjacent to boundary vertices. 
\item No trip uses the same edge twice, once in each direction, unless that trip starts (and ends) at a boundary vertex connected to a leaf.
\item No trips $T_1$ and $T_2$ share two edges $e_1,e_2$ such that $e_1$ comes before $e_2$ in both trips.  
\end{enumerate}

If $G$ is a reduced graph, each fixed point of $\sigma_G$ corresponds to a boundary leaf \citep{Lam13a}.
Suppose $G$ has $n$ boundary vertices, and suppose we have
\begin{equation}k=|\{a \in [n] \mid \sigma_G(a) < a \text{ or $\sigma_G(a)=a$ and $G$ has a white boundary leaf at $a$}\}|\end{equation}
Then we can construct a bounded affine permutation $f_G \in \Bd(k,n)$ corresponding to $G$ by setting
\begin{equation}f_G(a) = 
\begin{cases}
\sigma_G(a) & \sigma_G(a) > a \text{ or $G$ has a black boundary leaf at $a$}\\
\sigma_G(a)+n & \sigma_G(a) < a \text{ or $G$ has a white boundary leaf at $a$}
\end{cases}
\end{equation}
Thus we have have a correspondence between plabic graphs and positroid varieties; to a reduced plabic graph $G$, we associate the positroid $\Pi_G$ corresponding to $f_G$.  This correspondence is not a bijection.  Rather, we have a family of reduced plabic graphs for each positroid variety \citep{Lam13a}.  

We now describe a way to build plabic graphs inductively by adding new edges, called \emph{bridges}, to existing graphs.  The resulting graphs are called \emph{bridge graphs}.  This construction appears in \citep{ABCGPT14}, and also (in slightly less general form) in \citep{Lam13a}.  

We begin with a plabic graph $G$.  To add a bridge, we choose a pair of boundary vertices $a < b$, such every 
\begin{math} c \in [a+1,b-1]\end{math}
is a lollipop.  Our new edge will have one vertex on the leg at $a$, and one on the leg at $b$.  If $a$ (respectively $b$) is a lollipop, then the leaf at $a$ must be white (respectively black), and we use that boundary leaf as one endpoint of the bridge.  If $a$ (respectively $b$) is not a lollipop, we instead insert a white (black) vertex in the middle of the leg at $a$ (respectively, $b$).  We call the new edge an $(a,b)$-\emph{bridge}. After adding the new edge, our graph may no longer be bipartite.  In this case, we insert additional vertices of degree two or change the color of boundary vertices as needed to obtain a bipartite graph $G'$.  (See Figure \ref{addbridge}.)  If $G$ admits an almost perfect matching, then $G'$ does as well, so this process yields a plabic graph.  

The following proposition has been part of the plabic graph folklore for some time \citep{ABCGPT14, Lam13a}.  For completeness, we include a careful proof.

\begin{prop}\label{canadd}
Suppose $G$ is reduced.  Choose 
\begin{math} 1 \leq a < b \leq n\end{math}
such that
\begin{math} f_G(a) > f_G(b),\end{math}
and each 
\begin{math}c \in [a+1,b-1]\end{math}
is a lollipop.  Let $G'$ be the graph obtained by adding an $(a,b)$-bridge to $G$.  Then $G'$ is reduced, and 
\begin{displaymath} f'_G = f \circ (a,b) \in \Bd(k,n).\end{displaymath}  
Moreover, 
\begin{math}f'_G \lessdot f_G\end{math}
in the Bruhat order on $\Bd(k,n)$, and so $\Pi_G$ is a codimension-one subvariety of $\Pi_{G'}$.  
\end{prop}

\begin{proof}
For $i \in [n]$, let $T_i$ be the trip in $G$ which starts at vertex $i$, and let $T'_i$ be the corresponding trip in $G'$.  Then $T_i'$ coincides with $T_i$ for $i \neq a,b$. 
The trip $T'_a$ traverses first the $a$ leg, then the $(a,b)$-bridge, and then follows the path in $G'$ corresponding to the tail of $T_b$ after the leg at $b$.  Similarly, $T'_b$ traverses the $b$ leg and the $(a,b)$-bridge, then follows the tail corresponding to $T_a$.  Note that the trips $T_i'$ cover each edge of $G'$ exactly twice; there are no additional paths or cycles in the trip decomposition. Hence, to show that $G'$ is reduced, it is enough to show that $T'_a$ and $T_b'$ do not share two edges $e_1$ and $e_2$ that appear in the same order in both trips, and that neither $T'_a$ nor $T'_b$ use the same edge twice.
In particular, it suffices to show that the trips $T_a$ and $T_b$ in $G$ have no common edges.

Two trips in a plabic graph may touch at a common vertex without crossing; they cross if and only if they share a common edge.  Similarly, a trip crosses itself if and only if it traverses the same edge twice \citep[Section 13]{Pos06}.  Hence condition (3) in the definition of a reduced graph says that no trip crosses itself, except the trip at a boundary leaf, while condition (4) says that no two trips cross each other twice, with the two crossings occurring in the same order in both trips.

Since $f(a) > f(b),$ one of three things must hold
\begin{equation}\begin{cases}
f(a),f(b) \leq n& \sigma_G(a) > \sigma_G(b)\\
f(a),f(b) > n & \sigma_G(a) > \sigma_G(b)\\
f(a) > n,\,f(b) \leq n & \sigma_G(a) \leq a < b \leq \sigma_G(b)
\end{cases}\end{equation}
In Postnikov's language, we say $a$ and $b$ are \emph{aligned} \citep[Section 17]{Pos06}.
It follows from simple topological arguments that $T_a$ and $T_b$ cannot cross without violating condition (3) or (4).  Thus they have no common edges, as desired.  

It follows from the definitions that 
\begin{math}\sigma_{G'} = \sigma_G \circ (a,b).\end{math}
Since 
\begin{math} f_G(a) > f_G(b),\end{math}
and $f_G$ is a bounded affine permutation, we must have 
\begin{equation} a < f_G(b) < a+n \end{equation}
\begin{equation} b <  f_G(a) <  b+n.\end{equation}
Hence 
\begin{math} f_G \circ (a,b)\end{math}
is the unique bounded affine permutation of type $(k,n)$ corresponding to $\sigma_{G'}$, and $f_{G'} = f_G$.  The fact that 
\begin{math} f_G' \rightarrow f_G\end{math}
in Bruhat order follows from the definition of Bruhat order on $\Bd(k,n)$.  Since each 
\begin{math} c \in [a+1,b-1]\end{math}
is a fixed point of $\sigma_G$, while
\begin{equation}b \leq f_G(b) < f_G(a) \leq a+n,\end{equation}
it follows that either
\begin{equation} f_G(c) < f_G(a),f_G(b)\end{equation}
or
\begin{equation} f_G(c) > f_G(a),f_G(b) \end{equation}
for each 
\begin{math} c \in [a+1,b-1].\end{math}
Hence 
\begin{equation} \ell(f_{G'}) = \ell(f_G) - 1\end{equation} and so $f_{G'} \lessdot f_G$ as desired.  
\end{proof}

\begin{figure}[ht]
\centering
\begin{subfigure}[b]{\textwidth}
\centering
\begin{tikzpicture}[scale = 0.7]
\draw (2,4) -- (2,0);
\draw (0,3) -- (2,3); \wdot{0}{3}; \bdot{2}{3};
\draw (0,1) -- (2,1); \bdot{0}{1}; \wdot{2}{1};
\draw [->] (3,2) [out = 30, in = 150] to (5,2);
\begin{scope}[xshift = 6 cm]
\draw (2,4) -- (2,0);
\draw (0,3) -- (2,3); \wdot{0}{3}; \bdot{2}{3};
\draw (0,1) -- (2,1); \bdot{0}{1}; \wdot{2}{1};
\edge{0}{1}{2};
\end{scope}
\end{tikzpicture}
\caption{Adding a bridge between boundary leaves.}
\end{subfigure}
\\
\vspace{0.2 in}

\begin{subfigure}[b]{\textwidth}
\centering
\begin{tikzpicture}[scale = 0.7]
\draw (2,4) -- (2,0);
\draw (-1,3.5) -- (0,3);
\draw (-1,2.5) -- (0,3);
\draw (-1,1.5) -- (0,1);
\draw (-1,0.5) -- (0,1);
\draw (0,3) -- (2,3); \wdot{0}{3}; \bdot{2}{3};
\draw (0,1) -- (2,1); \bdot{0}{1}; \wdot{2}{1};
\draw [->] (3,2) [out = 30, in = 150] to (5,2);
\begin{scope}[xshift = 8 cm]
\draw (2,4) -- (2,0);
\draw (-2,3.5) -- (-1,3);
\draw (-2,2.5) -- (-1,3);
\draw (-2,1.5) -- (-1,1);
\draw (-2,0.5) -- (-1,1);
\draw (-1,3) -- (2,3); \wdot{-1}{3}; \bdot{2}{3};
\draw (-1,1) -- (2,1); \bdot{-1}{1}; \wdot{2}{1};
\edge{0.5}{1}{2};
\draw [->] (3,2) [out = 30, in = 150] to (5,2);
\end{scope}
\begin{scope}[xshift = 16 cm]
\draw (2,4) -- (2,0);
\draw (-2,3.5) -- (-1,3);
\draw (-2,2.5) -- (-1,3);
\draw (-2,1.5) -- (-1,1);
\draw (-2,0.5) -- (-1,1);
\draw (-1,3) -- (2,3); \wdot{-1}{3}; \bdot{2}{3};
\draw (-1,1) -- (2,1); \bdot{-1}{1}; \wdot{2}{1};
\edge{0.5}{1}{2}; \bdot{-0.25}{3}; \wdot{-0.25}{1};
\end{scope}
\end{tikzpicture}
\caption{Adding a bridge between legs which are not boundary leaves.  Note that after adding the bridge, we add additional vertices of degree $2$ to create a bipartite graph.}
\end{subfigure}
\caption{Adding bridges to a plabic graph.}
\label{addbridge}
\end{figure}
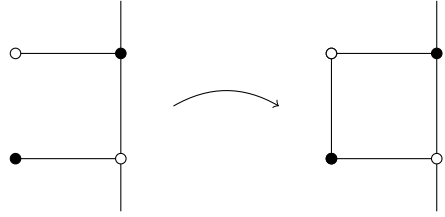
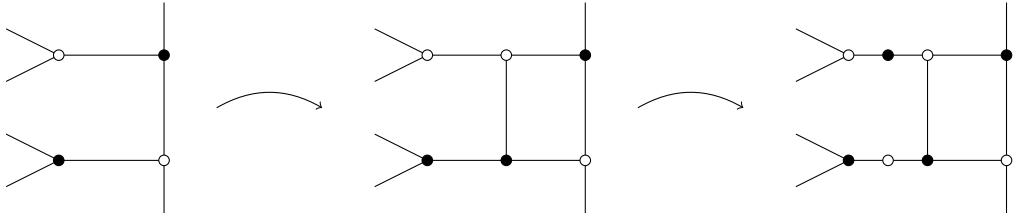

The zero-dimensional positroid varieties correspond to the points in $\Gr(k,n)$ which have a single non-zero Pl\"{u}cker coordinate $\mu$.  There is a unique reduced plabic graph for each $\mu \in {{[n]}\choose{k}}$, which has $n$ lollipops.  The $k$ lollipops corresponding to elements of $\mu$ are white; the rest are black.  We call a plabic graph consisting only of lollipops a \emph{lollipop graph}.  

A \emph{bridge graph} is a plabic graph which is constructed from a lollipop graph by successively adding bridges 
\begin{equation}(a_1,b_1),\ldots,(a_d,b_d),\end{equation} 
where at each step, $a_i,b_i$ is an inversion of the bounded affine permutation corresponding to the graph obtained by adding the first $i-1$ bridges.
By the preceding proposition, a bridge graph is always reduced.  
Let $u \leq_k w$.  Then by \eqref{permfactors} we have
\begin{equation}f_{u,w}= t_{u([k])}uw^{-1}\end{equation}
where $t_{u([k])}$ is the translation element corresponding to $u([k])$.
To construct a bridge graph for 
\begin{math}\Pi_{\langle u,w\rangle_k}\end{math}
we begin with the lollipop graph corresponding to $u([k])$, and successively add bridges to obtain a graph with bounded affine permutation $f_{u,w}$.  It is perhaps not obvious that every positroid variety has a bridge graph.  However, this follows both from our results, and from earlier work on the subject.  In particular, Postnikov's Le-diagrams correspond to a particular choice of bridge graph for each bounded affine permutation.

\subsection{Parametrizations from plabic graphs}

Let $G$ be a reduced plabic network with $e$ edges, and assign weights 
\begin{math}t_1,\ldots,t_e\end{math}
to the edges of $G$.  Postnikov defined a surjective map from the space of positive real edge weights of $G$ to the positroid cell 
\begin{math}\left(\mathring{\Pi}_G\right)_{\geq 0}\end{math}
in $\Gr_{\geq 0}(k,n)$, called the \emph{boundary measurement map} \citep[Section 11.5]{Pos06}.   Postnikov, Speyer and Williams re-cast this construction in terms of almost perfect matchings \citep[Section 4-5]{PSW09}, an approach Lam developed further in \citep{Lam13a}.  Muller and Speyer showed that we can apply the same map to the space of nonzero complex edge weights, and obtain a map to the positroid variety 
\begin{math} \mathring{\Pi}_G\end{math}
in $\Gr(k,n)$ \citep{MS14}.  We use the definition of the boundary measurement map found in \citep{Lam13a}. 

For $P$ an almost perfect matching on a plabic graph $G$ with $e$ edges, let 
\begin{equation}\partial(P) = \{\text{black boundary vertices used in $P$}\} \cup \{\text{white boundary vertices \emph{not} used in $P$}\}\end{equation}
We define the \emph{boundary measurement map} 
\begin{equation}D_G: \complexes^e \rightarrow \mathbb{P}^{{n \choose k}-1}\end{equation}
to be the map which sends 
\begin{math}(t_1,\ldots,t_e)\end{math}
to the point with homogeneous coordinates 
\begin{equation}\Delta_J = \sum_{\partial(P)=J} t^P\end{equation}
where the sum is over all matchings $P$ of $G$, and $t^P$ is the product of the weights of all edges used in $P$.

For positive real edge weights, the boundary measurement map is surjective onto \begin{math}\left( \mathring{\Pi}_G\right)_{\geq 0}\end{math}
but not necessarily injective.  Specializing all but a suitably chosen set of $d$ edge weights to $1$, and letting the remaining edge weights range over $\reals^+$, we obtain a homeomorphism from 
\begin{math}(\reals^+)^d\end{math}
to the positroid cell 
\begin{math}\left( \mathring{\Pi}_G\right)_{\geq 0}\end{math}
in the totally nonnegative Grassmannian.  If instead we let the edge weights range over 
\begin{math} \complexes^{\times},\end{math}
we obtain a map to the open positive variety 
\begin{math} \mathring{\Pi}_G\end{math}
in $\Gr(k,n)$. This map is well-defined on all of 
\begin{math} (\complexes^{\times})^d,\end{math}
and the image is an open dense subset of $\Pi_G$.  Specializing all but an appropriate set of edges to $1$ gives not a homeomorphism, but a birational map \citep{MS14}.

If $G$ is a bridge graph, there is a natural specialization of edge weights, and we have a simple procedure for constructing the desired parametrization. 
Let 
\begin{math} \Pi_G=\Pi_{\langle u, w \rangle_k},\end{math}
and let 
\begin{math}d = \dim (\Pi_G).\end{math} 
 Assign a weight 
\begin{math} t_1,\ldots,t_d\end{math}
to each bridge, in the order the bridges were added, and set all other edge weights to $1$.  Begin with the $k \times n$ matrix in which the columns indexed by $u([k])$ form a copy of the identity, while the remaining columns consist of all $0$'s.  Say the $r^{th}$ bridge is from $a_r$ to $b_r$, with $a_r < b_r$.  When we add the $r^{th}$ bridge to the graph, we multiply our matrix on the right by 
\begin{math} x_{(a_r,b_r)}(\pm t_r),\end{math}
the elementary matrix with nonzero entry $\pm t_r$ in row $a_r$ and column $b_r$.
The sign is negative if
\begin{math} |u([k]) \cap [a_r+1,b_r-1]|\end{math}
is odd, and positive if 
\begin{math} |u([k]) \cap [a_r+1,b_r-1]|\end{math}
is even.

\subsubsection{Example}

Let $n=4$, $k=2$.  Let $w=4321$ and $u = 2143$. Then 
\begin{math} u \leq_2 w,\end{math}
and 
\begin{math} \langle u, w \rangle_2\end{math}
corresponds to the big cell in $\Gr(2,4)$.  We have 
\begin{math} u([2]) = \{1,2\}.\end{math}  The bounded affine permutation corresponding to 
\begin{math} \langle u,w \rangle_2\end{math}
 is given by 
 \begin{math} f_{u,w} = 3456.\end{math}  Starting with the plabic graph corresponding to 
 \begin{math} t_{u([2])} = 5634\end{math}
we successively add bridges 
\begin{equation}(1,4) \rightarrow (2,4) \rightarrow (2,3) \rightarrow (1,2)\end{equation}
and obtain a sequence of bounded affine permutations
\begin{equation}5634 \rightarrow 4635 \rightarrow 4536 \rightarrow 4356 \rightarrow 3456.\end{equation}
The corresponding bridge graph is shown in Figure \ref{bridgegraph}.

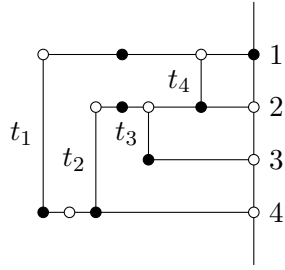
\begin{figure}[ht]
\centering
\begin{tikzpicture}[scale = 0.7]
\draw (4,0) -- (4,5);
\draw (4,1) -- (0,1);
\draw (4,2) -- (2,2);
\draw (4,3) -- (1,3);
\draw (4,4) -- (0,4);
\edge{0}{1}{3}; \node[left] at (0,2.5) {$t_1$};
\edge{1}{1}{2}; \node[left] at (1,2) {$t_2$};
\edge{2}{2}{1}; \node[left] at (2,2.5) {$t_3$};
\edge{3}{3}{1}; \node[left] at (3,3.5) {$t_4$};
\wdot{0.5}{1};
\wdot{4}{1};
\wdot{4}{2};
\wdot{4}{3};
\bdot{4}{4};
\bdot{1.5}{4};
\bdot{1.5}{3};
\foreach \x in {1,...,4}
	{\pgfmathtruncatemacro{\label}{5 - \x}
	 \node [right] at (4.1, \x) {\label};} 
\end{tikzpicture}
\captionsetup{singlelinecheck=off}
\caption[.]{This weighted bridge graph corresponds to the sequence of bridges 
\begin{displaymath}(1,4) \rightarrow (2,4) \rightarrow (2,3) \rightarrow (1,2).\end{displaymath}
}
\label{bridgegraph}
\end{figure}

We next construct the corresponding matrix parametrization.  Since 
\begin{math}u([2]) = \{1,2\},\end{math}
we start with the matrix which has a copy of the identity in its first two columns, and multiply on the right by a factor 
\begin{math} x_{(a_i,b_i)}(\pm t_i)\end{math}
for each bridge $(a_i,b_i)$.  So we have
\begin{equation}
\begin{bmatrix}
1 & 0 & 0 & 0\\
0 & 1 & 0 & 0
\end{bmatrix}
\begin{bmatrix}
1 & 0 & 0 & -t_1\\
0 & 1 & 0 & 0 \\
0 & 0 & 1 & 0 \\
0 & 0 & 0 & 1
\end{bmatrix}
\begin{bmatrix}
1 & 0 & 0 & 0\\
0 & 1 & 0 & t_2 \\
0 & 0 & 1 & 0 \\
0 & 0 & 0 & 1
\end{bmatrix}
\begin{bmatrix}
1 & 0 & 0 & 0\\
0 & 1 & t_3 & 0 \\
0 & 0 & 1 & 0 \\
0 & 0 & 0 & 1
\end{bmatrix}
\begin{bmatrix}
1 & t_4 & 0 & 0 \\
0 & 1 & 0 & 0 \\
0 & 0 & 1 & 0 \\
0 & 0 & 0 & 1
\end{bmatrix}
=
\begin{bmatrix}
1 & t_4 & 0 & -t_1\\
0 & 1 & c_3 & t_2 
\end{bmatrix}
\end{equation}

\subsection{Distinguished subexpressions}

\emph{Distinguished subexpressions} were first introduced by Deodhar, who used them to define a class of cell decompositions of the flag variety \citep{Deo85}.  Let $\mf{w}=s_{i_1}\cdots s_{i_m}$ be a reduced word for a permutation $w \in S_n$.  We gather some facts about distinguished subexpressions, borrowing most of our conventions from \citep[Section 3.6]{TW13}.  A \emph{subexpression} $\mf{u}$ of $\mf{w}$ is obtained by replacing some of the factors $s_{i_j}$ of $\mf{w}$ with the identity permutation, which we denote by $1$.  

We write \begin{math}\mf{u} \preceq \mf{w}\end{math} to indicate that $\mf{u}$ is a subexpression of $\mf{w}$.  We denote the $t^{th}$ factor of $\mf{u}$,  which may be either $1$ or a simple transposition, by $u_t$, and write $u_{(t)}$ for the product \begin{math}u_1u_2\ldots u_t\end{math}.  For notational convenience, we set \begin{math}u_0 = u_{(0)} = 1\end{math}.  We denote the the $t^{th}$ simple transposition in $\mf{u}$ by $u^t$.  

\begin{defn}A subexpression $\mf{u}$ of $\mf{w}$ is called \textbf{distinguished} if we have
\begin{equation}u_{(j)} \leq u_{(j-1)}s_{i_j}\end{equation}
for all \begin{math}1 \leq j \leq m\end{math}.
\end{defn}

\begin{defn}A distinguished subexpression $\mf{u}$ of $\mf{w}$ is called \textbf{positive distinguished} if 
\begin{equation}u_{(j-1)} < u_{(j-1)} s_{i_j}\end{equation}
for all \begin{math}1 \leq j \leq m\end{math}
We will sometimes abbreviate the phrase ``positive distinguished subexpression'' to \emph{PDS}.
\end{defn}

Given a subexpression \begin{math}\mf{u} \preceq \mf{w}\end{math}, we say $\mf{u}$ is \emph{a subexpression for} \begin{math}u= u^1u^2\cdots u^r \end{math}.  By abuse of notation, we identify the subexpression $\mf{u}$ with the word \begin{math} u^1u^2\cdots u^{r}\end{math}, also denoted $\mf{u}$.  If the subexpression $\mf{u}$ is positive distinguished, then the corresponding word of $u$ is reduced.

The following lemma is an easy consequence of the above definitions \citep[Lemma 3.5]{MR04}.

\begin{lem}
Let \begin{math} u \leq w \in S_n,\end{math} and let $\mf{w}$ be a reduced word for $w$.  Then the following are equivalent
\begin{enumerate}
\item $\mf{u}$ is a positive distinguished subexpression of $\mf{w}$,
\item $\mf{u}$ is the lexicographically first subexpression for $u$ in $\mf{w}$, working from the right.
\end{enumerate}
In particular, there is a unique PDS for $u$ in $\mf{w}$.
\end{lem}

Condition 2 means precisely that the factors $u^t$ are chosen greedily, as follows.  Suppose $\ell(u) = r$.  Working from the right, we set \begin{math}u^r = s_{i_j},\end{math} where $j$ is the largest index such that \begin{math}s_{i_j} \leq_{(l)} u.\end{math}  We then take $u^{r-1}$ to be the next rightmost factor of $\mf{w}$ such that \begin{math}u^{r-1}u^r \leq_{(l)} u\end{math} and so on, until we have \begin{math} u^1u^2\cdots u^r = u\end{math}.  

\subsection{Wiring diagrams, PDS's, and bridge diagrams}
\label{diagrams}

\emph{Wiring diagrams} provide a way to represent words in $S_n$ visually.  For convenience, we implicitly view all wiring diagrams as embedded in a coordinate system, with the $x$-coordinate increasing from left to right and the $y$-coordinate increasing from top to bottom.

Let $w$ in $S_n$, and fix a word $\mf{w}$ of $w$.  The wiring diagram for $\mf{w}$ has $n$ wires which run from left to right, with some number of crossings between them.  No two crossings have the same $x$-coordinate, and no more than two wires are involved in each crossing.  We number the right and left endpoints of the wires respectively from top to bottom.

Each crossing between two wires represents a simple transposition $s_i$ of $\mf{w}$, where $i-1$ is the number of wires in the diagram which pass directly above the crossing.  However, the crossings in the diagram for $\mf{w}$ appear in the opposite order as the simple generators in the word $\mf{w}$; the leftmost generator in $\mf{w}$ corresponds to the rightmost crossing in the diagram.  With these conventions, if $w(s)=t$, the wire with left endpoint $s$ has right endpoint $t$. A wiring diagram is \emph{reduced} if no two wires cross more than once; this occurs if and only if the word $\mf{w}$ is reduced. For an example of a reduced wiring diagram, see Figure \ref{wiringB}.  

We now define \emph{bridge diagrams}, which will be an essential tool in our proof of Theorem \ref{main}.  Let $\mf{y}$ be a word in $S_n$, not necessarily reduced, and let $\mf{x}$ be a subword of $\mf{y}$.  (We reserve the term \emph{subexpression} and the symbol $\preceq$ for subwords of reduced words.)  Then we may obtain a wiring diagram for $\mf{y}$ by starting with a wiring diagram for $\mf{x}$, and inserting new crossings corresponding to the remaining letters of $\mf{y}$.  To draw the \emph{bridge diagram} corresponding to the subword $\mf{x}$ of $\mf{y}$, we take a wiring diagram for $\mf{x}$, and draw dashed crosses between adjacent wires to represent the additional crossings from $\mf{y}$.  We call these dashed crosses \emph{bridges}, by analogy with bridge graphs.  For an example, see Figure \ref{wiringA}.  In the case where $\mf{y}$ is a reduced word, and $\mf{x}$ a subexpression of $\mf{y}$, we will sometimes write \begin{math} \mf{x} \preceq \mf{y}\end{math} to denote the bridge diagram corresponding to this subexpression.  Note that bridge diagrams are distinct from bridge graphs.  However, the proof of our main result shows that the two are intimately related.

\begin{figure}[ht]
\begin{subfigure}[t]{0.4\textwidth}
\begin{tikzpicture}[scale = 0.8]
\foreach \x in {0,...,3}
	{\pgfmathtruncatemacro{\label}{4 - \x}
	 \node [left]  at (0,\x) {\label};
	 \node [right] at (6, \x) {\label};} 
\btm{0}{0}; \brd{0}{1}; \tp{0}{2};
\btm{1}{0}; \btm{1}{1}; \crsg{1}{2};
\btm{2}{0}; \brd{2}{1}; \tp{2}{2};
\crsg{3}{0}; \tp{3}{1}; \tp{3}{2};
\btm{4}{0}; \brd{4}{1}; \tp{4}{2};
\btm{5}{0}; \btm{5}{1}; \brd{5}{2};
\end{tikzpicture}
\caption{The solid lines give the wiring diagram for $\mathbf{u}$.  The dashed crosses represent bridges.}
\label{wiringA}
\end{subfigure}
\hspace{0.5 in}
\begin{subfigure}[t]{0.4\textwidth}
\begin{tikzpicture}[scale = 0.8]
\foreach \x in {0,...,3}
	{\pgfmathtruncatemacro{\label}{4 - \x}
	 \node [left]  at (0,\x) {\label};
	 \node [right] at (6, \x) {\label};} \btm{0}{0}; \crsg{0}{1}; \tp{0}{2};
\btm{1}{0}; \btm{1}{1}; \crsg{1}{2};
\btm{2}{0}; \crsg{2}{1}; \tp{2}{2};
\crsg{3}{0}; \tp{3}{1}; \tp{3}{2};
\btm{4}{0}; \crsg{4}{1}; \tp{4}{2};
\btm{5}{0}; \btm{5}{1}; \crsg{5}{2};
\end{tikzpicture}
\caption{Replacing each bridge at left with a crossing gives the wiring diagram for $\mathbf{w}$.}
\label{wiringB}
\end{subfigure}
\\

\begin{subfigure}[t]{0.4 \textwidth}
\begin{tikzpicture}[scale = 0.8]
\foreach \x in {0,...,3}
	{\pgfmathtruncatemacro{\label}{4 - \x}
	 \node [left]  at (0,\x) {\label};
	 \node [right] at (6, \x) {\label};} 
\btm{0}{0}; \dmr{0}{1}; \tp{0}{2};
\btm{1}{0}; \btm{1}{1}; \crsg{1}{2};
\btm{2}{0}; \dmr{2}{1}; \tp{2}{2};
\crsg{3}{0}; \tp{3}{1}; \tp{3}{2};
\btm{4}{0}; \dmr{4}{1}; \tp{4}{2};
\btm{5}{0}; \btm{5}{1}; \dmr{5}{2};
\end{tikzpicture}
\caption{To construct the bridge graph corresponding to \begin{math}\mf{u} \preceq \mf{w},\end{math} we first replace each bridge with a dimer, as shown above.}
\label{wiringC}
\end{subfigure}
\hspace{0.5 in}
\begin{subfigure}[t]{0.4\textwidth}
\begin{tikzpicture}[scale = 0.8]
\foreach \x in {0,...,3}
	{\pgfmathtruncatemacro{\label}{4 - \x}
	 \node [right] at (6, \x) {\label};} 
\hdmr{0}{1};
\btm{1}{1}; \bcrsg{1}{2};
\bdmr{2}{1}; \tp{2}{2};
\tcrsg{3}{0}; \tp{3}{1}; \tp{3}{2};
\btm{4}{0}; \tdmr{4}{1}; \tp{4}{2};
\btm{5}{0}; \btm{5}{1}; \dmr{5}{2};
\end{tikzpicture}
\caption{Next, we delete the tail of each wire up the first dimer on that wire.  Adding degree-2 vertices as needed yields the desired plabic graph.}
\label{wiringD}
\end{subfigure}
\captionsetup{singlelinecheck=off}
\caption[.]{Constructing a bridge graph from a bridge diagram.  
In this example, \begin{math}\mathbf{w} = s_1s_2s_3s_2s_1s_2\end{math}, and \begin{math}\mathbf{u} = s_3s_1\end{math}.  In Section \ref{mainproof}, we will show that this process yields a planar graph whenever \begin{math}u \leq_k w\end{math} and \begin{math}\mf{u} \preceq \mf{w}\end{math} is a PDS.}
\label{wiring}
\end{figure}
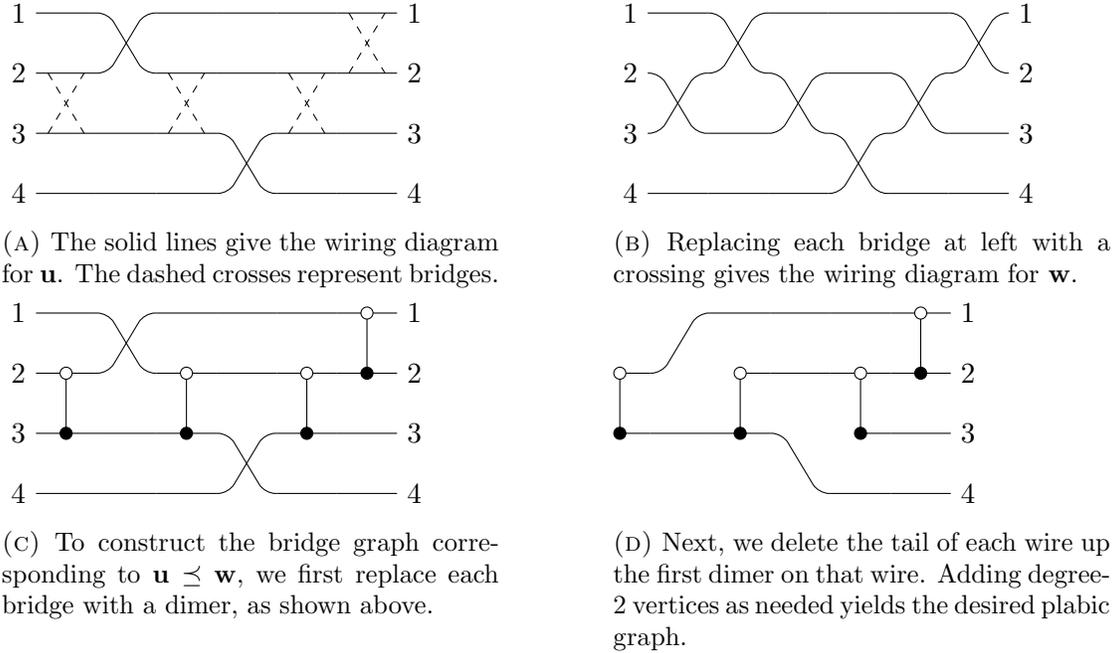

We now translate the notion of a positive distinguished subexpression into the language of bridge diagrams.
Suppose $\mf{w}$ is reduced, and let $\mf{u}$ be the PDS for $u$ in $\mf{w}$. 
We may construct the bridge diagram \begin{math} \mf{u} \preceq \mf{w}\end{math} from the (reduced) wiring diagram $\mf{u}$ by successively inserting bridges, working from 
left to right.

\begin{rmk} \label{nocross} Let $\mf{w}$ be a reduced word, let \begin{math}\mf{u} \preceq \mf{w},\end{math} and suppose $\mf{u}$ is itself reduced.  Consider the bridge corresponding to \begin{math}\mf{u} \preceq \mf{w}.\end{math}  The statement that $\mf{u}$ is a PDS means precisely that each bridge is inserted between two wires in the diagram for $\mf{u}$ which never cross again to the right of the bridge.   See Figure \ref{wiringA} for an example.
\end{rmk}

Indeed, suppose a bridge corresponds to the $j^{th}$ simple transposition $s_{i_j}$ in $\mf{w}$.  Then the fact that the bridge is inserted between wires that do not cross again is equivalent to the fact that \begin{math}u_{(j)} < u_{(j-1)}s_{i_j}.\end{math}

\begin{prop}\label{induce}
Let \begin{math} \mf{u} \preceq \mf{w}\end{math} and consider the corresponding bridge diagram.  Suppose we insert only the leftmost $r$ bridges in the diagram for $\mf{u}.$  (Here ``leftmost" is defined in terms of the diagram, not the word.)  Then replacing these bridges with crossings gives a reduced wiring diagram $\mf{v}$ for some $v \in S_n$, and the underlying diagram $\mf{u}$ represents the PDS for $u$ in $\mf{v}$. \end{prop}

\begin{proof}
The fact that $\mf{v}$ is reduced follows from Remark \ref{nocross}, together with the fact that $\mf{w}$ is reduced.  Since $\mf{v}$ is reduced, applying Remark \ref{nocross} again gives the rest of the proposition.
\end{proof}

We introduce some conventions for discussing bridge diagrams.  A bridge diagram is \emph{reduced} if it represents a reduced subword $\mf{u}$ of a reduced word $\mf{w}$.  We say that one crossing occurs \emph{before} another if the first crossing has a smaller $x$-coordinate than the second.  Wires in the underlying diagram for $\mf{u}$ are labeled by their right endpoints, so ``wire $a$" means the wire with right endpoint $a$. Suppose wires $a$ and $b$ cross in the diagram for $\mf{u}$, and suppose wire $a$ lies above $b$ to the left of the crossing, and below $b$ to the right.  We say wire $a$ crosses wire $b$ ``from above," and write \begin{math}(a \downarrow b).\end{math} Similarly, we say $b$ crosses wire $a$ ``from below," and write \begin{math}(b \uparrow a).\end{math}  We say wire $a$ is \emph{isolated} if there are no bridges touching wire $a$.  We refer to a bridge between wires $a$ and $b$ as an $(a,b)$-bridge.  If $a < b$, we call $a$ the \emph{upper wire} of the $(a,b)$-bridge, and call $b$ the \emph{lower wire}.  Note that we may also refer to bridges $(a,b)$ where $a > b$, and $a$ is the lower wire of the bridge.

We use the symbol $\rightarrow$ to indicate that one bridge or crossing occurs before (that is, to the left of) another.  So for example, \begin{math} (a,b) \rightarrow (c,d)\end{math} means there is a $(c,d)$ bridge after the $(a,b)$ bridge.  Similarly, \begin{math} (a,b) \rightarrow (c \downarrow b)\end{math} means that wire $c$ crosses wire $b$ from above after the $(a,b)$-bridge.

\subsection{Deodhar parametrizations of positroid varieties}

We now review Deodhar's decompositions of the flag variety, as well as Marsh and Rietsch's parametrizations of Deodhar components, and their projections to the Grassmannian \citep{Deo85, MR04}.  While these constructions are defined for general flag varieties, we focus on the case of $\mathcal{F}\ell(n)$. Our discussion follows \citep[Section 4]{TW13}, which in turn draws on \citep[Sections 1 and 3]{MR04}.  However, we use a slightly different set of conventions for our matrix representatives.

To construct a Deodhar decomposition of $\mathcal{F}\ell(n)$, we first fix a reduced word $\mf{w}$ for each $w \in S_n$.  These choices determine the decomposition \citep[Section 4]{MR04}. There is a Deodhar component \begin{math}\mathcal{R}_{\mf{u},\mf{w}}\end{math} for each distinguished subexpression \begin{math}\mf{u} \preceq \mf{w}.\end{math}  For a reduced word $\mf{w}$ of $w$, we have
\begin{equation}\mathring{X}_u^w = \bigsqcup_{\mf{u} \preceq \mf{w} \text{ distinguished}}\mathcal{R}_{\mf{u},\mf{w}}.\end{equation}

If \begin{math}u \leq_k w,\end{math} the projection \begin{math} \pi_k:\mathcal{F}\ell(n) \rightarrow \Gr(k,n)\end{math} is an isomorphism on the open Richardson variety \begin{math} \mathring{X}_u^w\end{math} whose image is the open positroid variety \begin{math} \mathring{\Pi}_{\langle u,w\rangle_k}\end{math} \citep{KLS13}.  To define a Deodhar decomposition of the Grassmannian, choose a representative \begin{math}\langle u,w \rangle_k\end{math} for each \begin{math} \langle u,w \rangle_k \in \mathcal{Q}(k,n),\end{math} and a reduced word $\mf{w}$ for each chosen $w$.  
For each of the selected \begin{math} u \leq_k w\end{math}, and each distinguished subexpression \begin{math} \mf{u} \preceq \mf{w},\end{math} the \emph{Deodhar component} of $\Gr(k,n)$ corresponding to \begin{math}\mf{u} \preceq \mf{w}\end{math} is given by 
\begin{equation} \mathcal{D}_{\mf{u},\mf{w}} := \pi_k(\mathcal{R}_{\mf{u},\mf{w}}).\end{equation}

Composing the parametrization of 
\begin{math} \mathcal{R}_{\mf{u},\mf{w}}\end{math} from \citep{MR04} with the projection $\pi_k$ gives a parametrization of \begin{math}\mathcal{D}_{\mf{u},\mf{w}}\end{math}.  For \begin{math}\mf{u} \preceq \mf{w}\end{math} a PDS, the component \begin{math} {D}_{\mf{u},\mf{w}}\end{math} is a dense open subset of \begin{math}\mathring{\Pi}_{\langle u,w\rangle_k},\end{math} so we obtain a parametrization of this positroid variety.  

We denote the elementary matrix with non-zero entry $t$ at position $(i,j)$ by \begin{math}x_{(i,j)}(t)\end{math} if $i < j$, and by \begin{math}y_{(i,j)}(t)\end{math} if $i > j.$ 
For \begin{math} 1 \leq i \leq n-1,\end{math} we write $\dot{s_i}$ for the matrix obtained from the $n \times n$ identity by replacing the $2 \times 2$ block whose upper left corner is at position $(i,i)$ with the block matrix

\begin{equation}\begin{bmatrix}
0 & -1\\
1 & 0
\end{bmatrix}\end{equation}

Let \begin{math} \mf{u} \preceq \mf{w},\end{math} where \begin{math}\mf{w} = s_{i_1}\cdots s_{i_m}.\end{math}  We define $g_j$ for \begin{math}1 \leq j \leq m\end{math} as follows, where the $t_j$ and $p_j$ are parameters.

\begin{equation}g_j =\left\{ 
\begin{array}{ll}
x_{i_j}(t_j) & \text{if $s_{i_j}$ is not in $\mf{u}$}\\
\dot{s_{i_j}}^{-1} & \text{if $s_{i_j}$ is in $\mf{u}$ and $\ell(u_{(j)}) > \ell(u_{(j-1)})$}\\
\dot{s_{i_j}}y_{i_j}(p_j) & \text{if $s_{i_j}$ is in $\mf{u}$ and $\ell(u_{(j)})<\ell(u_{(j-1)})$}
\end{array} 
\right\}\end{equation}
Note that if $\mf{u}$ is a PDS of $\mf{w}$, the third case never occurs. Next, we define the set
\begin{equation}G_{\mf{u},\mf{w}} :=\{ g_{m}g_{m-1}\ldots g_1 \in \text{GL}(n) \mid t_j \in \complexes^{\times}, p_j \in \complexes\}\end{equation}
For \begin{math} u = w = 1\end{math}, we set \begin{math} \mf{G_{u,w}} = 1.\end{math}  
The projection from \begin{math} G_{\mf{u},\mf{w}}\end{math} to $\mathcal{F}\ell(n)$ gives a homeomorphism onto the Deodhar component corresponding to \begin{math}\mf{u} \preceq \mf{w}.\end{math}  Let 
\begin{equation}s = |\{j \mid \text{$s_{i_j}$ is in $\mf{u}$ and $\ell(u_{(j)})<\ell(u_{(j-1)})$}\}|\end{equation}
\begin{equation}r = |\{j \mid  \text{$s_{i_j}$ is not in $\mf{u}$}\}|\end{equation}
Then the obvious map \begin{math}\complexes^s \times (\complexes^{\times})^r \rightarrow G_{\mf{u},\mf{w}}\end{math} gives a parametrization of the Deodhar component \begin{math} \mathcal{R}_{\mf{u},\mf{w}}\end{math} \citep{MR04}.

As mentioned previously, our conventions differ from those of \citep{MR04} and \citep{TW13}.  In particular, Talaska and Williams view $\mathcal{F}\ell(n)$ as $GL(n)$ modulo the right action of the group of upper-triangular matrices, while we quotient by the left action of the lower-triangular matrices.  Hence our matrix parametrizations are transposed with respect to theirs.  In addition, we let $\dot{s_i}$ represent the matrix with a nonzero entry at position $(i,i+1)$, rather than at position $(n-i,n-i+1)$.  This corresponds to interchanging the roles of $k$ and $n-k$.
Despite these superficial differences, our set-up is essentially equivalent to the one in \citep{TW13}.

The matrices $\dot{s_i}$ and $x_{(a,b)}$ satisfy the following relation, which may be checked directly.
\begin{lem}
Suppose \begin{math} s_i(a) < s_i(b).\end{math}  Then
\begin{equation} \label{signs} \dot{s_i}x_{(a,b)}(t)\dot{s_i}^{-1}=
\begin{cases}
x_{(s_i(a),s_i(b))}(-t)& i \in \{a-1,b-1\}\\
x_{(s_i(a),s_i(b))}(t)& \text{otherwise}
\end{cases}
\end{equation}
\end{lem}

Let \begin{math} \mf{u} \preceq \mf{w}\end{math} be a PDS, where \begin{math} \ell(w) = m.\end{math} For \begin{math}1 \leq j \leq m,\end{math} let 
\begin{equation}\dot{u_j}=\begin{cases}
\dot{s_{i_j}} & s_{i_j} \in \mf{u}\\
1 & \text{otherwise}
\end{cases}\end{equation}

 Let \begin{math}d = \ell(w)-\ell(u),\end{math} and let \begin{math}j_1,\ldots,j_d\end{math} be the indices corresponding to simple transpositions which are not in $\mf{u}$.  
 For $1 \leq r \leq d$, define
 \begin{equation}\bar{r} = d+1-r\end{equation}
  and set
 \begin{equation}\beta_{\bar{r}} = (\dot{u_1}\cdots \dot{u}_{j_r-1}) (x_{i_{j_r}}(t_{j_r}))(\dot{u}_{j_r-1}^{-1} \cdots \dot{u_1}^{-1}).\end{equation}
 Then we can rewrite each \begin{math} G \in G_{\mf{u},\mf{w}}\end{math} in the form
\begin{equation}\label{betas}G = (\dot{u}_m^{-1}\dot{u}_2^{-1} \cdots \dot{u}_1^{-1})(\beta_1\beta_2 \cdots \beta_d)\end{equation}

 \begin{lem}
 For $\beta_r$ as above, define:
 \begin{align} a &= u_{(j_r-1)}(i_{j_r})\\
 b &= u_{(j_r-1)}(i_{j_r}+1)\\
 \theta &= |u([k]) \cap [a+1,b-1])|
 \end{align}
 Then we have
 \begin{equation} \beta_{\bar{r}} = x_{(a,b)}((-1)^{\theta} t_{j_r}).\end{equation}
 \end{lem}
 
 \begin{proof}
 Since $\mf{u}$ is a PDS of $\mf{w}$, we have 
 \begin{equation} \ell(u_{(j_r-1)}s_{i_{j_r}}) > \ell(u_{(j_r-1)}).\end{equation}
 From standard Coxeter arguments, it follows that for each \begin{math}1 \leq q \leq m,\end{math} we have
 \begin{equation}u_q(u_{q+1}\cdots u_{j_r-1}(i_{j_r})) < u_q(u_{q+1}\cdots u_{j_r-1}(i_{j_r}+1))\end{equation}
 Hence we can apply \eqref{signs} repeatedly, to obtain \begin{math}\beta_{\bar{r}} = x_{(a,b)}(\pm t_{j_r}).\end{math}
 
Next, we compute the sign of the parameter \begin{math} \pm t_{j_r}.\end{math}  In the language of bridge diagrams, \eqref{signs} implies that
we multiply the parameter $t_{j_r}$ by a factor of $-1$ for each wire $c$ in the diagram for $\mf{u}$ such that \begin{math}(c \downarrow a)\end{math} after the $(a,b)$ bridge; and one for each wire such $c'$ such that \begin{math}(c' \downarrow b)\end{math} after the $(a,b)$ bridge.   After canceling, this yields a factor of $-1$ for each wire which crosses wire $a$ from above after the $(a,b)$ bridge, and whose right endpoint is between $a$ and $b$.  By Lemma \ref{escape} below, these are precisely the wires $c$ with \begin{math}c \in \{u([k]) \cap [a+1,b-1]\}\end{math}.  The claim follows. \end{proof}

For notational convenience, we renumber our parameters $t_{i_j}$ and define $a_r,b_r$ such that 
\begin{equation} \beta_r = x_{(a_r,b_r)}(\pm t_r)\end{equation}
where the sign is determined as above.
It follows from the proposition that if the simple transposition $s_{i_{j_r}}$ corresponds to an $(a,b)$ bridge in the diagram for \begin{math}\mf{u} \preceq \mf{w}\end{math}, then \begin{math}\beta_{\bar{r}}=x_{(a,b)}(\pm t_{\bar{r}})\end{math}.  Note that the $\beta_r$ appear in \eqref{betas} in the same order as the corresponding bridges $(a_r,b_r)$ in the diagram for \begin{math}\mf{u} \preceq \mf{w}\end{math} (and hence in the opposite order as the factors $s_{i_{j_r}}$ in the word for $\mf{w}$).

\subsubsection{Example continued}

As before, let $n=4$, $k=2$.  Let \begin{math} \mf{w} = s_1s_2s_3s_2s_1s_2 \in S_4,\end{math} and let $u=2143$.   The positive distinguished subexpression $\mf{u}$ for $u$ in $\mf{w}$ comprises the $s_3$ in position $3$ from the left, and the $s_1$ in position $5$, so we have a parametrization of \begin{math} \mathcal{R}_{\mf{u},\mf{w}}\end{math} by
\begin{equation}G_{\mf{u},\mf{w}}=x_2(t_1)\dot{s_1}^{-1}x_2(t_2)\dot{s_3}^{-1}x_2(t_3)x_1(t_4).\end{equation}
Rewriting this in the form \eqref{betas} and projecting to $\Gr(2,4),$ we obtain
\begin{equation}\begin{bmatrix}
1 & 0 & 0 & 0\\
0 & 1 & 0 & 0
\end{bmatrix}
\dot{s_1}^{-1} \dot{s_3}^{-1}x_{(1,4)}(-t_1)x_{(2,4)}(t_2)x_2(t_3)x_1(t_4)=\begin{bmatrix}
0 & 1 & t_3 & t_2\\
-1 & -t_4 & 0 & t_1
\end{bmatrix}
\end{equation}

Re-ordering the rows and multiplying the first row by $-1$, gives
\begin{equation}\begin{bmatrix}
1 & t_4 & 0 & -t_1\\
0 & 1 & t_3 & t_2
\end{bmatrix}\end{equation}
which is precisely the matrix we obtained from a bridge decomposition corresponding to \begin{math}\mf{u} \preceq \mf{w}.\end{math}  So these two parametrizations yield the same point in the Grassmannian, for each choice of parameters \begin{math}(t_1,t_2,t_3,t_4).\end{math} 

\section{From PDS's to bridge graphs}
\label{mainproof}

We next outline a method which produces a bridge graph corresponding to a projected Deodhar parametrization.  While the method is straightforward, proving that it yields a  bridge graph requires considerable work.

Let \begin{math}\mf{w}=s_{i_1}s_{i_2}\cdots s_{i_m}\end{math} be a reduced word for $w$, let \begin{math}u \leq_k w,\end{math} and let \begin{math}\mf{u}\preceq \mf{w}\end{math} be the PDS for  $u$ in $\mf{w}$.  
Let \begin{math} j_1,\ldots, j_d\end{math} be the indices of the simple transpositions of $\mf{w}$ which are not in $\mf{u}$, and for \begin{math}1 \leq r \leq d,\end{math} let \begin{math}\bar{r}=d+1-r.\end{math}  For \begin{math}1 \leq r \leq d,\end{math} let 
\begin{equation}(a_{\bar{r}},b_{\bar{r}})=u_{j_r-1}(s_{i_{j_r}})u_{j_r-1}^{-1}.\end{equation}
Then we have
\begin{equation}(a_{\bar{1}},b_{\bar{1}})(a_{\bar{2}},b_{\bar{2}})\cdots(a_{\bar{d}},b_{\bar{d}})=wu^{-1}.\end{equation}
Reversing the order of the transpositions gives
\begin{equation} \label{gens} (a_1,b_1)(a_2,b_2)\cdots (a_d,b_d)=uw^{-1}\end{equation}
and so we have
\begin{equation} t_{u([k])}(a_1,b_1)(a_2,b_2)\cdots (a_d,b_d)=f_{u,w} \label{rightperm} \end{equation}

To construct the desired bridge graph, we successively add bridges 
\begin{displaymath}(a_1,b_1),\cdots,(a_d,b_d)\end{displaymath}
to the lollipop graph corresponding to $u([k])$.  This is possible, so long as the hypotheses of Proposition \ref{canadd} are satisfied at each step.  For the moment, let us assume this is the case; that is, the sequence of transpositions in \eqref{gens} corresponds to a bridge graph $G$.  It follows from \eqref{rightperm} that 
\begin{equation}\Pi_G=\Pi_{\langle u,w \rangle_k}.\end{equation}  

We claim that the parametrization arising from $G$ is precisely the projected Deodhar parametrization corresponding to 
\begin{math} \mf{u} \preceq \mf{w}.\end{math}  For this, recall that the matrices $\beta_r$ from the previous section are given by 
\begin{displaymath} \beta_r = x_{(a_r,b_r)}(\pm t_r)\end{displaymath}
for all \begin{math}1 \leq r \leq d.\end{math}  Hence, we may construct both parametrizations by taking the matrix which has a single non-zero Pl\"ucker coordinate $u([k])$, and multiplying on the right by a sequence of factors \begin{math}x_{(a_r,b_r)}(\pm t_r).\end{math}  In each case, the sign of the parameter is negative if \begin{math}|u([k]) \cap [a_r+1,b_r-1]|\end{math} is odd, and positive otherwise, so the parametrizations are the same.  

\begin{rmk}\label{threethings} Our task is to prove that the sequence of transpositions in \eqref{gens} will always correspond to a bridge graph.  There are a number of things to check. In particular, suppose that adding the bridges
\begin{equation}(a_1,b_1),\ldots,(a_{r-1},b_{r-1}),\end{equation}
yields a reduced bridge graph corresponding to some bounded affine permutation $f_{r-1}$.  To prove that we can add the bridge $(a_r,b_r)$, we must show the following:
\begin{enumerate}
\item \begin{math} f_{r-1}(a_r) < f_{r-1}(b_r)\end{math}
\item If $a_r$ (respectively $b_r$) is a fixed point of $f_{r-1}$, then the boundary leaf at $a_r$ is white (respectively, black).
\item Each $c$ with \begin{math} a_r< c < b_r\end{math} is a lollipop.
\end{enumerate}
\end{rmk}

In the remainder of this section, we show that these criteria are always satisfied, so that we can build a bridge graph corresponding to each projected Deodhar parametrization.

\subsection{Anti-Grassmannian permutations, $k$-order, and wiring diagrams}

We now reduce to the case where $u$ is \emph{anti-Grassmannian}.  This, in turn, allows us to prove the main result by induction on \begin{math}\ell(w)-\ell(u).\end{math}

\begin{lem}\label{antiG} Let \begin{math}u \leq_k w.\end{math} Then there exists \begin{math}z \in S_k \times S_{n-k}\end{math} such that \begin{math}uz \leq_k wz\end{math}, where $uz$ is anti-Grassmannian and both factorizations are length additive.  
\end{lem}

\begin{proof}We prove this by descending induction on the length of $u$.  If $u$ is of maximal length in its left coset of \begin{math}S_k \times S_{n-k},\end{math} then $z$ is the identity.  Otherwise, there is some $i \neq k$ such that $u(i) < u(i+1)$.  By Theorem \ref{kcomp}, \begin{math}w(i) < w(i+1)\end{math}, and \begin{math}us_i \leq_k ws_i.\end{math}  By induction, this means that there exists \begin{math}z' \in S_k \times S_{n-k}\end{math} such that \begin{math}(us_i)z'\leq_k (ws_i)z',\end{math} with both factorizations length additive.  Hence, we can take \begin{math}z = s_iz',\end{math} and the proof is complete.
\end{proof} 

Let $u,w$ and $z$ be as above.  Choose a reduced word $\mf{w}$ for $w$ and a reduced word $\mf{z}$ for $z$, and let $\mf{u}$ be the unique PDS for $u$ in $\mf{w}$.
Then concatenating $\mf{u}$ and $\mf{z}$ gives the unique PDS
$\mf{uz}$ for $uz$ in the reduced word \begin{math}\mf{w}\mf{z}.\end{math}  
It is clear, however, that  the pairs \begin{math}\mf{u} \preceq \mf{w}\end{math} and \begin{math} \mf{uz} \preceq \mf{wz}\end{math} project to the same Deodhar parametrization of \begin{math}\Pi_{\langle u, w \rangle_k}.\end{math}
Hence, we can always assume that $\mf{u}$ is an anti-Grassmannian permutation.  

\begin{defn}\label{valid} A bridge diagram is \emph{valid} if it corresponds to a pair $\mf{u} \preceq \mf{w}$ with $u$ anti-Grassmannian, and $\mf{u}$ a PDS of $\mf{w}$.  \end{defn}

We will prove the main result by induction on the number of bridges.  By Proposition \ref{induce}, adding the first $r$ bridges in the diagram corresponding to a PDS \begin{math}\mf{u} \preceq \mf{w}\end{math} gives a bridge diagram for some \begin{math}\mf{u} \preceq \mf{v},\end{math} where $\mf{v}$ is a reduced word for some $v \in S_n$ and $\mf{u}$ is a PDS of $\mf{v}$.  However, our result does not hold for all PDS's, only those corresponding to pairs \begin{math}u,w \in S_n\end{math} with $u \leq_k v$.  Hence, we must show that we always have \begin{math}u \leq_k v.\end{math}  Since $u$ is anti-Grassmannian, this is an easy lemma.

\begin{lem}\label{korder}
If $u$ is anti-Grassmannian, then $w \geq u$ implies \begin{math} w \geq_k u\end{math}.  
\end{lem}

\begin{proof}
Since $u$ is anti-Grassmannian, the second condition of Theorem \ref{kcomp} is vacuously true.  Hence, it suffices to show that the first condition holds.  That is, we must show that \begin{math}u(a) \leq w(a)\end{math} for $a \leq k$, and \begin{math}u(a) \geq w(a)\end{math} for $a > k$.

We recall the usual criterion for comparison in Bruhat order.  That is, $u \leq w$ if and only if for all $a \in [n]$, we have \begin{math}u([a]) \leq w([a])\end{math}.  For $u$ anti-Grassmannian, it follows that \begin{math}u(a) \leq w(a)\end{math} for all $a \leq k$.  Indeed, for $a \leq k$, we have \begin{math}u(a) = \min(u[a])\end{math}.  If \begin{math}w(a) < u(a),\end{math} then \begin{math}\min(w[a]) < \min(u[a]),\end{math} a contradiction.

Next, we must show that \begin{math}w(a) \leq u(a)\end{math} for all $a > k$.  For $a > k$, $u(a)$ is the largest element of \begin{math}[n] \backslash u([a-1]).\end{math}  In particular, each $b$ with \begin{math}u(a) < b \leq n\end{math} is in $u([a-1])$.  Since we have \begin{math}w([a-1]) \geq u([a-1]),\end{math} each \begin{math}u(a) < b \leq n\end{math} is in $w([a-1])$ is well.  It follows that \begin{math}w(a) \leq u(a),\end{math} and the proof is complete.
\end{proof}

\subsection{Planarity of the graph}

Consider a valid bridge diagram \begin{math} \mf{u} \preceq \mf{w} \end{math}, as defined in Definition \ref{valid}.  We describe a way to construct the corresponding bridge graph.   For an example, see Figure \ref{wiring}.  We think of the right endpoints of the wires as boundary vertices of a bicolor graph embedded in a disk; the wires themselves as paths with one endpoint on the boundary;  and the new crossings as white-black bridges between these paths.  Ignore the tail of each wire from the left endpoint up to the first bridge on that wire.  Add a white boundary leaf at the right endpoint of each isolated wire with left endpoint $\leq k$, and a black boundary leaf at the right endpoint of each isolated wire with left endpoint $> k$.  

We say that a bridge diagram is \emph{planar} if the above process yields a planar embedding of a bridge graph.  If this occurs, the graph must necessarily be the bridge graph for \begin{math}\Pi_{\langle u, w \rangle_k}\end{math} described previously, with its sequence of bridges given by \eqref{gens}.  Hence, our goal is to prove that a valid bridge diagram will always be planar.  

 \begin{lem}\label{last}  Let $a$ be a non-isolated wire in a valid bridge diagram $\mf{u} \prec \mf{w}$, and let $t$ be the left endpoint of wire $a$.  If the first bridge on wire $a$ is a bridge $(a,b)$, with $b > a$, then $t \leq k$.  If instead $b < a$, then $t > k$.
 \end{lem}

\begin{proof}We argue the first case; the proof of the second case is analogous.  Note that wire $b$ must lie below wire $a$ to the right of the $(a,b)$ bridge.  Suppose the first bridge on wire $a$ is $(a_r,b_r)$.  Let $w'$ be the permutation corresponding to the wiring diagram obtained by adding bridges 
\begin{displaymath}(a_d,b_d),\ldots,(a_r,b_r)\end{displaymath}
to the diagram for $\mf{u}$ and replacing each bridge with a crossing.  Then 
\begin{equation} w'(t) = b > a = u(t)\end{equation}
Since $w' \geq_k u$, Theorem \ref{kcomp} implies that $t \leq k$. 
\end{proof}

\begin{rmk}
As a corollary, note that the second condition in Remark \ref{threethings} will always be satisfied, as long as the sequence of bridges 
\begin{equation} (a_1,b_1),\ldots,(a_d,b_d)\end{equation}
is planar.  This follows from the lemma, since white lollipops correspond to isolated wires whose left endpoint is $\leq k$, and black lollipops correspond to isolated wires whose right endpoint is $>k$.
\end{rmk}

\begin{lem} \label{escape}Let $a,b$ and $c$ be wires in a valid bridge diagram, with $b < c$.  Let $t$ be the left endpoint of wire $a$.  If we have \begin{math}(b,c) \rightarrow (a \downarrow b)\end{math} or \begin{math}(b,c) \rightarrow (a \downarrow c),\end{math}
then $t \leq k$.  If \begin{math}(b,c) \rightarrow (a \uparrow b)\end{math} or \begin{math}(b,c) \rightarrow (a \uparrow c)\end{math}, then $t > k$.
\end{lem}

\begin{proof}

We prove the case where \begin{math}(a \downarrow b)\end{math} or \begin{math}(a \downarrow c)\end{math}.  The other case is analogous.  If \begin{math}(b,c) \rightarrow (a \downarrow c),\end{math} then we must have 
\begin{displaymath} (b,c) \rightarrow (a \downarrow b) \rightarrow (a \downarrow c)\end{displaymath}
so it suffices to consider the case \begin{math}(a \downarrow b).\end{math}  See Figure \ref{escapelemma}. 

There are two cases to consider.  If wire $b$ satisfies the conditions of Lemma \ref{last}, then the left endpoint of $b$ is $\leq k$.  Since  \begin{math}(a \downarrow b),\end{math} the same is true of $a$. Otherwise, we must have \begin{math}(e,b) \rightarrow (b,c)\end{math} for some $e < b$.  Hence \begin{math}(e,b) \rightarrow (a \downarrow b),\end{math} and so \begin{math}(e,b) \rightarrow (a \downarrow e).\end{math}
Hence, we can apply the previous argument to $e$, and so on.  Eventually, we must encounter a wire $e'$ such that \begin{math}(a \downarrow e'),\end{math} and $e'$ satisfies the conditions of Lemma \ref{last}.  This completes the proof.

\end{proof}

\begin{figure}[ht]
\centering
\begin{tikzpicture}[scale = 0.9]
\draw (0,0) to [out = 0, in = 180] (3,0.5);
\draw (3.5,0.5) to [out = 0, in = 180] (6.5,0);
\mbrd{3}{0.5}
\draw (0,1.5) to [out = 0, in = 180] (3,1);
\draw (3.5,1) to [out = 0, in = 180] (6.5,1.5);
\draw [dashed] (0,2.5) to [out = 0, in = 180] (3,1.75) to [out = 0, in = 180] (6.5,-1);
\sdot{6.5}{0};\sdot{6.5}{-1};\sdot{6.5}{1.5};
\node [right] at (6.5,0) {$c$};
\node [right] at (6.5,1.5) {$b$};
\node [right] at (6.5,-1) {$a$};
\end{tikzpicture}
\caption{If the left endpoint of wire $b$ is $\leq k$, the same must be true for wire $a$.}
\label{escapelemma}
 \end{figure}
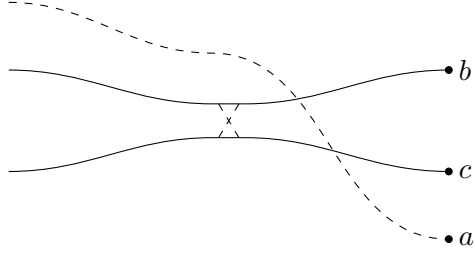
 
Let $\mf{u} \preceq \mf{w}$ be a valid bridge diagram, and suppose inserting the bridges 
\begin{displaymath}(a_1,b_1),\ldots,(a_{r-1},b_{r-1})\end{displaymath}
from the diagram for \begin{math}\mf{u} \preceq \mf{w}\end{math} into the diagram for $\mf{w}$ gives a planar bridge diagram.  We will show that adding the bridge $(a_r,b_r)$ with \begin{math}a_r < b_r\end{math} preserves planarity. The proof consists of repeatedly apply lemmas \ref{last} and \ref{escape}.  There are three cases to consider, depending on whether wires $a_r$ and $b_r$ are isolated.

\begin{lem} \label{2wires}
If $a_r$ and $b_r$ are both non-isolated wires, then adding the bridge $(a_r,b_r)$ preserves planarity.
\end{lem}

\begin{proof}
Note that we add the $(a_r,b_r)$ bridge to the right of all previous ones, and that wire $b_r$ lies immediately below wire $a_r$ at the location of the $(a_r,b_r)$ bridge.  By inductive assumption, the portions of wires $a_r$ and $b_r$ respectively to the right of all the bridges 
\begin{displaymath}(a_1,b_1),\ldots,(a_{r-1},b_{r-1})\end{displaymath}
correspond to legs of a plabic graph, as described above.  Hence, adding the $(a_r,b_r)$ bridge corresponds to adding a new edge between two adjacent boundary legs of a plabic graph, and the result is planar.  See Figure \ref{noniso}.
\end{proof}

\begin{figure}[ht]
\centering
\begin{subfigure}[t]{0.4\textwidth}
\centering
\begin{tikzpicture}[scale = 0.9]
\draw [white] (0,-2) [rectangle] (6.7,3);
\draw (0,1.75) to [out = 0, in = 180] (2,1.5) to (2.5,1.5) to [out = 0, in = 180] (6.5,1.75);
\draw (0,0.75) to [out = 0, in = 180] (2,1) to (2.5,1) to [out = 0, in = 180] (5,0.75) to (5.5,0.75) to [out = 0, in = 180] (6.5,0.85);
\draw (0,1) to [out = 0, in = 180] (3,0) to (3.5,0) to [out = 0, in = 180] (5,0.25) to (5.5,0.25) to [out = 0, in = 180] (6.5,0.15);
\draw (0,-1) to [out = 0, in =180] (3,-0.5) to (3.5,-0.5) to [out = 0, in = 180] (6.5,-1);
\dbrd{2}{1}; \dbrd{3}{-0.5};\dbrd{5}{0.25};
\sdot{6.5}{1.75};\sdot{6.5}{0.85};\sdot{6.5}{0.15};\sdot{6.5}{-1};
\node [right] at (6.5,0.85) {$a_r$}; \node [right] at (6.5,0.15) {$b_r$};
\end{tikzpicture}
\caption{A portion of a bridge diagram which contains the rightmost bridge $(a_r,b_r)$.}
\end{subfigure}
\hfill
\begin{subfigure}[t]{0.4 \textwidth}
\centering
\begin{tikzpicture}[scale = 0.9]
\draw (2,0) -- (2,5); \draw (0,1) -- (2,1); \draw (0,2) -- (2,2); \draw (-1,3) -- (2,3); \draw (-1,4) -- (2,4);
\edge{0}{1}{1};\edge{-1}{3}{1};\edge{1}{2}{1};
\wdot{2}{1};\wdot{2}{2};\bdot{2}{3};\bdot{2}{4};
\node [right] at (2,3) {$a_r$};
\node [right] at (2,2) {$b_r$};
\end{tikzpicture}
\caption{The portion of a plabic network corresponding to the bridge diagram at left.}
\end{subfigure}
\caption{Adding a rightmost bridge $(a_r,b_r)$ between two non-isolated wires corresponds to adding a bridge between two adjacent legs of a planar network.}
\label{noniso}
\end{figure}
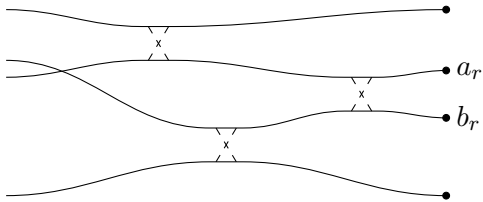
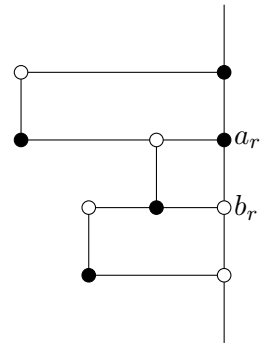

\begin{lem} \label{1wire}
Suppose exactly one of the wires $(a_r,b_r)$ is isolated.  Then adding the $(a_r,b_r)$ bridge gives a planar diagram. 
\end{lem}

\begin{proof}
We argue the case where $b_r$ is the isolated wire.  The other case is analogous.  By Lemma \ref{last}, the left endpoint of wire $b_r$ must be $> k$.  

Suppose there is a non-isolated wire $c$ with \begin{math}a_r < c < b_r,\end{math} and let $(c,c')$ be a bridge on wire $c$. Since $a_r$ and $c$ are both non-isolated, the inductive assumption implies that wire $c$ lies below $a_r$ to the right of the bridges 
\begin{displaymath}(a_1,b_1),\ldots,(a_{r-1},b_{r-1})\end{displaymath}
already inserted.  In particular, wire $c$ lies below wire $a_r$ at the horizontal position where we insert the $(a_r,b_r)$ bridge, and hence below wire $b_r$ as well.  Hence, we must have 
\begin{displaymath}(c,c') \rightarrow (a_r,b_r) \rightarrow (b_r \downarrow c).\end{displaymath}  Since the left endpoint of wire $b_r$ is $>k$, this contradicts lemma \ref{escape}.

We have shown there is no non-isolated wire $c$ with \begin{math}a_r < c < b_r.\end{math}  This, together with the inductive assumption and that fact that $a_r$ is non-isolated, ensures that adding that the $(a_r,b_r)$-bridge preserves the planarity condition.  See Figure \ref{cases}.

\end{proof}

\begin{figure}[ht]
\centering
 \begin{subfigure}[t]{0.4 \textwidth}
 \centering
 \begin{tikzpicture}[scale = 0.9]
 \draw (1.5,-0.35) to [out = 0, in = 180] (2.25,-0.5) to (2.75,-0.5) to [out = 0, in = 180] (6,0.5)
 to (6.5,0.5) to [out = 0, in = 180] (8,1);
 \dbrd{2.25}{-1};
 \dbrd{6}{0};
 \draw (1.5,-1.15) to [out = 0, in = 180] (2.25,-1) to (2.75,-1) to [out = 0, in = 180] (8,-2.5);
 \draw (1.5,-2) to [out = 0, in = 180] (6,0) to (6.5,0) to [out = 0, in = 180] (8,-0.5);
 \draw [dashed] (1.5,-2.5) to [out = 0, in = 180] (5.25,-0.75) to (5.75,-0.75) to [out = 0, in = 180] (8,0);
 \draw [dashed] (1.5,-3.25) to [out = 0, in = 180] (5.25,-1.25) to (5.75,-1.25) to [out = 0, in = 180] (8,-2);
 \dbrd{5.25}{-1.25};
 \sdot{8}{1}; \sdot{8}{0}; \sdot{8}{-0.5}; \sdot{8}{-2}; \sdot{8}{-2.5};
 \node[right] at (8,1) {$a_r$}; \node[right] at (8,0) {$c$}; \node[right] at (8,-0.5) {$b_r$};
 \node[right] at (8,-2) {$c'$}; \node[right] at (8,-2.5) {$e$};
 \end{tikzpicture}
 \caption{The case $e > b_r$}
 \end{subfigure}
 \hfill
 \begin{subfigure}[t]{0.4 \textwidth}
 \centering
 \begin{tikzpicture}[scale = 0.9]
\draw (1.5,0.25) to [out = 0, in = 180] (3,0.5) to (3.5,0.5) to [out = 0, in = 180] (5,0.25) to [out = 0, in = 180] (8,-0.25);
\draw (1.5,1.25) to [out = 0, in = 180] (3,1) to (3.5,1) to [out = 0, in = 180] (8,1.5);
\draw (1.5,-0.25) to [out = 0, in = 180] (3,0) to (5,-0.25) to (5.5,-0.25) to [out = 0, in = 180] (8,-1.5);
\draw[dashed] (3,-1.15) to [out = 0, in = 180] (4,-1.25) to (4.5,-1.25) to [out = 0, in = 180] (8,-0.75);
\draw [dashed] (3,-1.8) to [out = 0, in = 180] (4,-1.75) to (4.5,-1.75) to [out = 0, in = 180] (8,-2.25);
\dbrd{5}{-0.25};\dbrd{3}{0.5};\dbrd{4}{-1.75};
\sdot{8}{-0.75}; \sdot{8}{-2.25};\sdot{8}{-0.25}; \sdot{8}{1.5}; \sdot{8}{-1.5};
\node [right] at (8,-0.25) {$a_r$};
\node [right] at (8,1.5) {$e$};
\node [right] at (8,-1.5) {$b_r$};
\node [right] at (8,-0.75) {$c$}; 
\node [right] at (8,-2.25) {$c'$};
 \end{tikzpicture}
 \caption{The case $e < a_r$ and $c' > a_r$}
 \end{subfigure}
\\

 \begin{subfigure}[t]{0.4 \textwidth}
 \centering
 \begin{tikzpicture}[scale = 0.9]
\draw (1,1.35) to [out = 0, in = 180] (1.5, 1.25) to (2,1.25) to [out = 0, in = 180] (4,1) to (4.5,1) to [out = 0, in = 180] (8,1.5);
\draw (3.5,0.4) to [out = 0, in = 180] (4,0.5) to (4.5,0.5) to [out = 0, in = 180] (6,0.25) to (6.5,0.25) to [out = 0, in = 180] (8,-0.25);
\draw (4,0) to [out = 0, in = 180] (6,-0.25) to (6.5,-0.25) to [out = 0, in = 180] (8,-1.25);
\draw[dashed] (1,0.65) to [out = 0, in = 180] (1.5,0.75) to (2,0.75) to [out = 0, in = 180] (5,-1) to [out = 0, in = 180] (8,-0.75);
\dbrd{1.5}{0.75}; \dbrd{6}{-0.25};\dbrd{4}{0.5};
\sdot{8}{-0.75};\sdot{8}{-0.25}; \sdot{8}{1.5}; \sdot{8}{-1.25};
\node [right] at (8,-0.25) {$a_r$};
\node [right] at (8,1.5) {$e$};
\node [right] at (8,-1.25) {$b_r$};
\node [right] at (8,-0.75) {$c$}; 
 \end{tikzpicture}
 \caption{The case $e < a_r$ and $c' = e$}
 \end{subfigure}
 \hfill
  \begin{subfigure}[t]{0.4 \textwidth}
 \centering
 \begin{tikzpicture}[scale = 0.9]
\draw (3.5,1.1) to [out = 0, in = 180] (4,1) to (4.5,1) to [out = 0, in = 180] (8,1.5);
\draw (3.5,0.4) to [out = 0, in = 180] (4,0.5) to (4.5,0.5) to [out = 0, in = 180] (6,0.25) to (6.5,0.25) to [out = 0, in = 180] (8,-0.25);
\draw (4,0) to [out = 0, in = 180] (6,-0.25) to (6.5,-0.25) to [out = 0, in = 180] (8,-1.25);
\draw[dashed] (1,1.4) to [out = 0, in = 180] (1.5,1.25) to (2,1.25) to [out = 0, in = 180] (8,2.25);
\draw[dashed] (1,0.65) to [out = 0, in = 180] (1.5,0.75) to (2,0.75) to [out = 0, in = 180] (5,-1) to [out = 0, in = 180] (8,-0.75);
\dbrd{1.5}{0.75}; \dbrd{6}{-0.25};\dbrd{4}{0.5};
\sdot{8}{-0.75};\sdot{8}{-0.25}; \sdot{8}{1.5}; \sdot{8}{-1.25};\sdot{8}{2.25};
\node [right] at (8,-0.25) {$a_r$};
\node [right] at (8,1.5) {$e$};
\node [right] at (8,-1.25) {$b_r$};
\node [right] at (8,-0.75) {$c$}; 
\node [right] at (8,2.25) {$c'$};
 \end{tikzpicture}
 \caption{The case $e < a_r$ and $c' < a_r$}
 \end{subfigure}
\captionsetup{singlelinecheck=off}
\caption[.]{Adding a bridge $(a_r,b_r)$, where $b_r$ is an isolated wire, and there is a bridge $(a_r,e)$ for some $e$.  In each case, the existence of a wire $c$ with \begin{math}a_r < c < b_r\end{math} and a bridge $(c,c')$ forces $b_r$ to cross some non-isolated wire from above, to the right of all bridges on that wire.  This yields a contradiction.}
\label{cases}
 \end{figure}
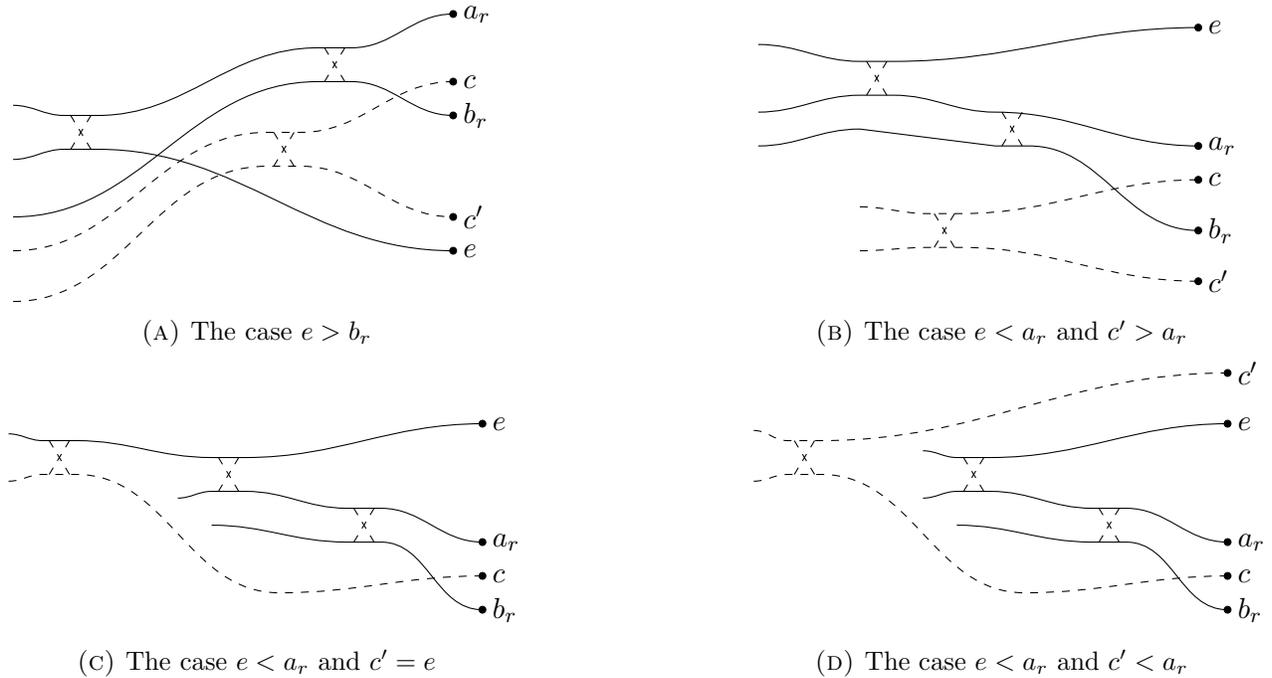

\begin{lem} \label{init}
If wires $a_r$ and $b_r$ are both isolated, then adding the bridge $(a_r,b_r)$ gives a planar diagram.
\end{lem}

\begin{proof}
By Lemma \ref{last}, the left endpoint of $a_r$ is $\leq k$, while the left endpoint of $b_r$ is $> k$.  Let $c$ be a non-isolated wire, so that we have a bridge $(c,c')$ to the left of the $(a_r,b_r)$ bridge.  Suppose toward a contradiction that \begin{math}a_r < c < b_r.\end{math}  Then either we have 
\begin{displaymath} (c,c') \rightarrow (a_r,b_r) \rightarrow (a_r \uparrow c)\end{displaymath}
or we have 
\begin{displaymath} (c,c') \rightarrow (a_r,b_r) \rightarrow (b_r \downarrow c).\end{displaymath}
In the first case, Lemma \ref{escape} implies that the left endpoint of $a_r$ is $> k$, and in the second, Lemma \ref{escape} implies that the left endpoint of $b_r$ is $\leq k$.  In either case, we have a contradiction, so each wire $c$ with \begin{displaymath}a_r < c < b_r\end{displaymath} is isolated.  See Figure \ref{iso}.

Next, suppose $c > b_r$.  To prove planarity, we must show that wire $b_r$ lies above wire $c$ to the right of the $(a_r,b_r)$ bridge.  It is enough to show that the $(a_r,b_r)$ bridge lies above wire $c$.  
Suppose the $(a_r,b_r)$ bridge lies below wire $c$.  Then we have 
\begin{displaymath}(c,c') \rightarrow (a_r,b_r) \rightarrow (a_r \uparrow c),\end{displaymath}
which gives a contradiction as before.

Next, let $c'$ be a wire which is not isolated, and suppose we have $c' < a_r$.  By an analogous argument, the $(a_r,b_r)$ bridge must be inserted below wire $c'$.  Hence, adding a bridge $(a_r,b_r)$ preserves the planarity condition, and the proof is complete.  \end{proof}

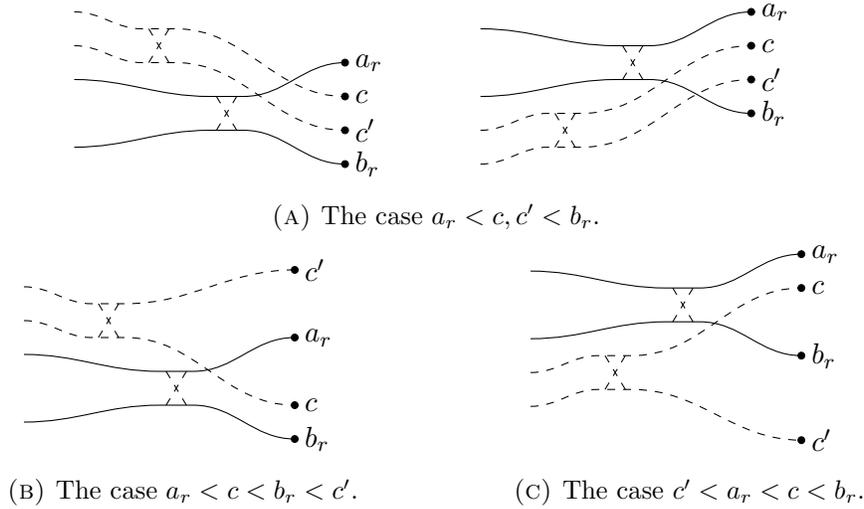
\begin{figure}[ht]
\centering 
\begin{subfigure}[t]{\textwidth}
\centering
\begin{tikzpicture}[scale = 0.9]
\draw[dashed] (0,1.75) to [out = 0, in = 180] (1,1.5) to (1.5,1.5) to [out = 0, in = 180] (4,0.5);
\draw[dashed] (0,2.25) to [out = 0, in = 180] (1,2) to (1.5,2) to [out = 0, in = 180] (4,1);
\draw (0,1.25) to [out = 0, in = 180] (2,1) to (2.5,1) to [out = 0, in = 180] (4,1.5);
\draw (0,0.25) to [out = 0, in = 180] (2,0.5) to (2.5,0.5) to [out = 0, in = 180] (4,0);
\dbrd{2}{0.5}; \dbrd{1}{1.5};
\node [right] at (4,0) {$b_r$}; \node [right] at (4,0.5) {$c'$}; \node [right] at (4,1) {$c$}; \node [right] at (4,1.5) {$a_r$};
\sdot{4}{0};\sdot{4}{0.5};\sdot{4}{1};\sdot{4}{1.5};
\begin{scope}[xshift = 6 cm, yshift = 0.75 cm]
\draw (0,1.25) to [out = 0, in = 180] (2,1) to (2.5,1) to [out = 0, in = 180] (4,1.5);
\draw (0,0.25) to [out = 0, in = 180] (2,0.5) to (2.5,0.5) to [out = 0, in = 180] (4,0);
\draw[dashed] (0,-0.25) to [out = 0, in = 180] (1,0) to (1.5,0) to [out = 0, in = 180] (4,1);
\draw[dashed] (0,-0.75) to [out = 0, in = 180] (1,-0.5) to (1.5,-0.5) to [out = 0, in = 180] (4,0.5);
\dbrd{2}{0.5}; \dbrd{1}{-0.5};
\node [right] at (4,0) {$b_r$}; \node [right] at (4,0.5) {$c'$}; \node [right] at (4,1) {$c$}; \node [right] at (4,1.5) {$a_r$};
\sdot{4}{0};\sdot{4}{0.5};\sdot{4}{1};\sdot{4}{1.5};
\end{scope}
\end{tikzpicture}
\caption{The case \begin{math}a_r < c,c' <b_r.\end{math}}
\end{subfigure}
\\
\begin{subfigure}[t]{0.4\textwidth}
\centering
\begin{tikzpicture}[scale = 0.9]
\draw[dashed] (0,2.25) to [out = 0, in = 180] (1,2) to (1.5,2) to [out = 0, in = 180] (4,2.5);
\draw[dashed] (0,1.75) to [out = 0, in = 180] (1,1.5) to (1.5,1.5) to [out = 0, in = 180] (4,0.5);
\draw (0,1.25) to [out = 0, in = 180] (2,1) to (2.5,1) to [out = 0, in = 180] (4,1.5);
\draw (0,0.25) to [out = 0, in = 180] (2,0.5) to (2.5,0.5) to [out = 0, in = 180] (4,0);
\dbrd{2}{0.5}; \dbrd{1}{1.5};
\node [right] at (4,0) {$b_r$}; \node [right] at (4,0.5) {$c$}; \node [right] at (4,1.5) {$a_r$}; \node [right] at (4,2.5) {$c'$}; 
\sdot{4}{0};\sdot{4}{0.5};\sdot{4}{1.5};\sdot{4}{2.5};
\end{tikzpicture}
\caption{The case \begin{math}a_r < c<b_r<c'.\end{math}}
\end{subfigure}
\begin{subfigure}[t]{0.4\textwidth}
\centering
\begin{tikzpicture}[scale = 0.9]
\draw (0,1.25) to [out = 0, in = 180] (2,1) to (2.5,1) to [out = 0, in = 180] (4,1.5);
\draw (0,0.25) to [out = 0, in = 180] (2,0.5) to (2.5,0.5) to [out = 0, in = 180] (4,0);
\draw[dashed] (0,-0.25) to [out = 0, in = 180] (1,0) to (1.5,0) to [out = 0, in = 180] (4,1);
\draw[dashed] (0,-0.75) to [out = 0, in = 180] (1,-0.5) to (1.5,-0.5) to [out = 0, in = 180] (4,-1.25);
\dbrd{2}{0.5}; \dbrd{1}{-0.5};
\node [right] at (4,0) {$b_r$}; \node [right] at (4,-1.25) {$c'$}; \node [right] at (4,1) {$c$}; \node [right] at (4,1.5) {$a_r$};
\sdot{4}{0};\sdot{4}{-1.25};\sdot{4}{1};\sdot{4}{1.5};
\end{tikzpicture}
\caption{The case \begin{math}c'<a_r < c <b_r.\end{math}}
\end{subfigure}\captionsetup{singlelinecheck=off}
\caption[.]{Adding a bridge $(a_r,b_r)$ between two isolated wires.  The existence of a wire $c$ with \begin{math}a_r < c < b_r\end{math} and a bridge $(c,c')$ forces either \begin{math}(c,c') \rightarrow (a_r \uparrow c)\end{math} or \begin{math}(c,c') \rightarrow (b_r \downarrow c)\end{math} which gives a contradiction.}
\label{iso}
\end{figure}

Combining lemmas $\ref{init},$ $\ref{2wires},$ and $\ref{1wire}$, we see that adding bridge $(a_r,b_r)$ always yields a planar bridge diagram.  By induction, we have the following result.

\begin{prop}
If \begin{math}\mf{u} \preceq \mf{w}\end{math} is a PDS, with $u$ anti-Grassmannian, than the bridge diagram for \begin{math}\mf{u} \preceq \mf{w}\end{math} is planar.  
\end{prop}

\subsection{Proving the graph is reduced}

Given a PDS \begin{math}\mf{u} \preceq \mf{w}\end{math} with $u$ anti-Grassmannian, we have shown that we can build a plabic graph by taking the lollipop graph $u([k])$, adding bridges 
\begin{displaymath}(a_1,b_1),\ldots,(a_d,b_d)\end{displaymath}
as specified by \eqref{gens}, and then inserting degree-two vertices as needed to make the graph bipartite. To show that this process yields a bridge graph, it remains to check the third condition of Remark \ref{threethings}.  Suppose that adding bridges 
\begin{displaymath}(a_1,b_1),\ldots,(a_{r-1},b_{r-1})\end{displaymath}
gives a reduced graph, with bounded affine permutation $f_{r-1}.$
We claim that \begin{math}f_{r-1}(a_r) > f_{r-1}(b_r).\end{math}

Let $\mf{v}$ be the reduced wiring diagram for some $v \in S_n$ obtained from $\mf{u}$ by replacing the bridges up to \begin{math}(a_{r-1},b_{r-1})\end{math} with crossings. Inserting the bridge $(a_r,b_r)$ at the appropriate place in the wiring diagram for $\mf{v}$ gives a reduced wiring diagram for some \begin{math}v' \gtrdot v,\end{math} so we must have \begin{math}v^{-1}(a_{r}) < v^{-1}(b_{r}).\end{math}  By construction, we have
\begin{equation} \label{image} f_{r-1}(c)=
\begin{cases}
uv^{-1}(c)+n& v^{-1}(c) \leq k\\
uv^{-1}(c) & \text{ otherwise}
\end{cases}
\end{equation}

If \begin{math} v^{-1}(a_r)  \leq k\end{math} and \begin{math}v^{-1}(b_r) > k\end{math}, the inequality \begin{math}f_{r-1}(a_r) > f_{r-1}(b_r)\end{math} follows easily from (\ref{image}).  Otherwise, either we have 
\begin{equation}v^{-1}(a_{r}),v^{-1}(b_{r}) \in \{1,\ldots,k\}\end{equation}
or we have
 \begin{equation}v^{-1}(a_{r}),v^{-1}(b_{r}) \in \{k+1,\ldots,n\}.\end{equation} 
 The claim then follows since $u$ is anti-Grassmannian, and is hence order-reversing on the sets \begin{math}\{1,\ldots,k\}\end{math} and \begin{math}\{k+1,\ldots,n\}.\end{math}

Hence, starting with the lollipop graph for $u([k])$ and adding bridges as in \eqref{gens} gives a reduced bridge graph for \begin{math}\Pi_{\langle u,w \rangle_k}.\end{math}  This proves the first direction of Theorem \ref{main}.

\begin{prop}
Every projected Deodhar parametrization for a positroid arises from a bridge graph.  
\end{prop}

\section{From bridge graphs to PDS's}
\label{converse}

We have shown that each projected Deodhar parametrization corresponds to a bridge graph.  Next, we show the converse: every parametrization arising from a bridge graph coincides with some projected Deodhar parametrization.

We introduce a bit more terminology for discussing bridge diagrams.  
We call the portion of a bridge diagram to the right of the leftmost bridge, including the bridge itself, the \emph{restricted part} of the diagram; we call the remainder of the diagram the \emph{free part}.  An $(a,b)$-\emph{junction} refers to either an $(a,b)$-bridge or a crossing between wires $a$ and $b$.

Suppose we have a bridge graph $G$ for some positroid variety in $\Gr(k,n).$  Adding a white lollipop to $G$ anywhere along the boundary yields a bridge graph for a positroid variety in \begin{math}\text{Gr}(k+1,n+1)\end{math} while adding a black lollipop yields a bridge graph for a positroid variety in $\text{Gr}(k,n+1)$.  

\begin{lem}\label{addfp}
Let $G$ be a bridge graph, and let $G'$ be a graph obtained from $G$ by adding a lollipop.  Suppose we have a bridge diagram corresponding to $G$.  Then we can construct a bridge diagram for $G'$.
\end{lem}

\begin{proof}

We argue the case of adding a black lollipop; the case of a white lollipop is analogous.  By assumption, we have a valid bridge diagram \begin{math}\mf{u} \preceq \mf{w}\end{math} corresponding to $G$, where $u,w \in S_n$.  Adding a black lollipop to $G$ and renumbering boundary vertices  gives a reduced bridge graph $G'$ for some \begin{math}\Pi_{\langle u',w' \rangle_k}\end{math} with $u'$ anti-Grassmannian.  Note that $u'$ and $w'$ are uniquely determined by the position of the new lollipop.

The restricted part of the diagram \begin{math}\mf{u} \preceq \mf{w}\end{math} is a reduced bridge diagram $B$ corresponding to some \begin{math} \mf{x} \preceq \mf{y}.\end{math}  We claim that we can add a new wire to  $B$ to produce a bridge diagram $B'$ for some $\mf{x'} \preceq \mf{y'}$, which we can then extend to a bridge diagram \begin{math}\mf{u'} \preceq \mf{w'} \end{math} corresponding to $G'$.

First, suppose the black lollipop is inserted just counterclockwise of position $1$, and the remaining lollipops are re-numbered $2$ through $n$.  Then we simply add a new $1$-wire which runs straight across the top of the diagram for $\mf{x}$, and renumber the endpoints of the existing wires appropriately.  Otherwise, the black lollipop is inserted just clockwise of position $q$, for some $q \geq 2$. We add a new right endpoint $q$ directly below $q-1$ in the bridge diagram for $\mf{x} \preceq \mf{y}$, and renumber the remaining right endpoints accordingly.  We then construct the diagram \begin{math}\mf{x}' \preceq \mf{y}'\end{math} in sections, working from right to left.

First, we divide the diagram $B$ into sections, as show in Figure \ref{addlolly}.  Let $c_0=q$.  Find the rightmost point where either $c_0$ crosses another wire $e_0$, or there is a bridge $(c_0,e_0)$ with $e_0 > c_0$.  
Let $B_0$ denote the portion of $B$ which begins just to the left of the $(c_0,e_0)$-junction, and extends to the rightmost boundary of $B$. 
Let $c_1$ denote either $c_0$ or $e_0$, whichever wire lies below the other immediately to the left of the $(c_0,e_0)$-junction.  

Now, suppose we have already defined sections \begin{math} B_0,\ldots,B_{i-1},\end{math} and fixed wire $c_i$. If there is no point to the left of $B_{i-1}$ in $B$ where either $c_i$ crosses another wire $e_i$, or there is a bridge $(c_i,e_i)$ with $e_i> c_i$, then let $B_i$ be the portion of $B$ to the left of $B_{i-1}$.  Otherwise, let $B_i$ be the portion of $B$ whose left edge is just to the left of the $(c_i,e_i)$-junction, and whose right edge is the boundary of $B_{i-1}$.  Let $c_{i+1}$ be either $c_i$ or $e_i$, whichever is lower to the left of the $(c_i,e_i)$-junction.  Continuing in this fashion, we divide all of $B$ into sections, as in Figure \ref{addlolly}.

Next, we modify each $B_i$ by adding a segment $q_i$ of wire $q$.  The path of $q_i$ depends on the nature of the $(c_i,e_i)$ junction, as shown in Figure \ref{segment}.  If \begin{math}(e_i \uparrow c_i),\end{math} we let $q_i$ lie  immediately below $e_i$ to the left of the crossing, and immediately below $c_i$ to the right of the crossing.  If \begin{math}(e_i \downarrow c_i),\end{math} then we let $q_i$ lie below $c_i$ to the left of the crossing; let \begin{math}(q_i \uparrow e_i)\end{math} immediately to the right of the crossing; and let $q_i$ run immediately below $c_i$ to the boundary of $B_i$.  Finally, if the $(e_i,c_i)$ junction is a bridge with $e_i > c_i$, we let $q_i$ start below wire $e_i$; let \begin{math}(q_i \uparrow e_i)\end{math} immediately to the right of the bridge; and let $q_i$ run directly below $c_i$ to the boundary of $B_i$.
In each case, we shift the wires below $q_i$ downward, to obtain a wiring diagram which satisfies our conventions.  See Figure \ref{segment} and Figure \ref{addlolly}.

\begin{figure}[ht]
 \centering
 \begin{subfigure}[t]{0.4\textwidth}
 \centering
 \begin{tikzpicture}[scale = 0.9]
 \draw [very thick] (-1,-0.5) to [out = 0, in = 225] (0,0) to [out = 45, in = 180] (3,1);
  \draw (-1,0.5) to [out = 0, in = 135] (0,0) to [out = 315, in = 180] (3,-1);
  \draw [dashed] (-1,-1) to [out = 0, in = 225] (0,-0.5) to [out = 45, in = 180] (3,0.5);
  \node [right] at (3,1) {$c_i$}; 
  \node [right] at (3,-1) {$e_i$}; 
  \node [right] at (3,0.5) {$q_i$};
   \end{tikzpicture}
 \caption{The case \begin{math}(e_i \downarrow c_i).\end{math}}
 \end{subfigure}
 \begin{subfigure}[t]{0.4\textwidth}
 \centering
 \begin{tikzpicture}[scale = 0.9]
  \draw (-1,-0.5) to [out = 0, in = 225] (0,0) to [out = 45, in = 180] (3,0.75);
  \draw [very thick] (-1,0.5) to [out = 0, in = 135] (0,0) to [out = 315, in = 180] (3,-0.75);
  \draw [dashed] (-1,-1) to [out = 0, in = 180] (0,-0.75) to [out = 0, in = 160] (0.5,-1) to [out = 340, in = 180] (3,-1.25);
  \node [right] at (3,0.75) {$e_i$}; 
  \node [right] at (3,-0.75) {$c_i$}; 
 \node [right] at (3,-1.25) {$q_i$};
 \end{tikzpicture}
 \caption{The case $(e_i \uparrow c_i)$.}
 \end{subfigure}
  \begin{subfigure}[t]{0.4\textwidth}
 \centering
\begin{tikzpicture}[scale = 0.9]
\draw [very thick] (-1,0.5) to [out = 0, in =180] (-0.125,0.25) to (0.25,0.25) to [out = 0, in = 180] (3,1);
\draw (-1,-0.5) to [out = 0, in = 180] (-0.125,-0.25) to (0.125,-0.25) to [out = 0, in = 180] (3,-1);
 \draw [dashed]  (-1,-1) to [out =0 , in = 180] (0,-0.75) to [out = 0, in = 180] (3,0.5);
 \dbrd{-0.125}{-0.25};
 \node [right] at (3,1) {$c_i$};
  \node [right] at (3,0.5) {$q_i$}; 
  \node [right] at (3,-1) {$e_i$};
 \end{tikzpicture}
 \caption{A $(e_i,c_i)$ bridge, $e_i > c_i$.}
 \end{subfigure}
 \caption{Adding a segment $q_i$ of wire $q$ to the bridge diagram $B_i$.  Here $q$ corresponds to a black lollipop.}
   \label{segment}
 \end{figure}
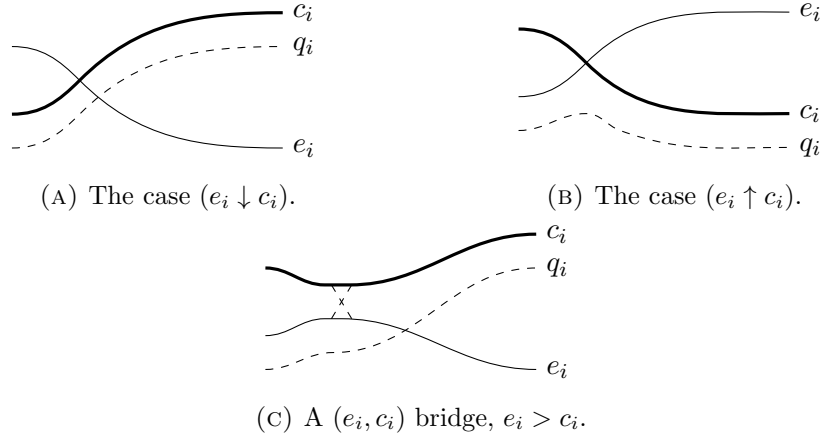
 
Adding the segment $q_i$ to each section $B_i$ of $B$ yields a bridge diagram $B'$ with underlying wiring diagram $\mf{x'}$; by construction, the $q_i$ form a unbroken wire $q$.  We claim that $B'$ is reduced.   Let $\mf{y'}$ be the wiring diagram we obtain from $\mf{x'}$ by adding the bridges inherited from the diagram $B$, and replacing each bridge with a crossing.  Since $B$ is reduced, it suffices to show that wire $q$ does not cross any wire more than once, in either $\mf{x'}$ or $\mf{y'}$.  This follows, since every crossing involving wire $q$ has the form \begin{math}(q \uparrow c)\end{math} for some $c$.

Hence, we have a reduced bridge diagram $B'$ corresponding to a subexpression \begin{math}\mf{x'} \preceq \mf{y'},\end{math} which we obtained from $B$ by adding an isolated wire $q$.   By construction, the sequence of bridges 
\begin{displaymath}(a_1,b_1),\ldots,(a_d,b_d)\end{displaymath}
in $B'$ is precisely the sequence of bridges in the graph $G'$.  It suffices to show that we can add additional crossings on the left side of $B'$ to create a valid bridge diagram \begin{math}\mf{u'} \preceq \mf{w'},\end{math} whose restricted part is $B'$.

The free part of the diagram \begin{math}\mf{u} \preceq \mf{w}\end{math} is a reduced wiring diagram $\mf{v}$ for some $v \in S_n$.  By construction, we have 
\begin{equation} x^{-1}u = y^{-1}w = v. \end{equation} 
Now, $u'$ and $w'$ are uniquely determined by $u$ and $w$; the fact that $u'$ is anti-Grassmannian; and the fact that
\begin{equation}u'^{-1}(q)=w'^{-1}(q)>k.\end{equation} It follows that 
\begin{equation}x'^{-1}u'=y'^{-1}w'=v'\end{equation} for some $v' \in S_{n+1}$.  
Consider the concatenation of a reduced diagram $\mf{v'}$ for $v'$ and the diagram $B'$.  We claim that this is the desired bridge diagram \begin{math}\mf{u'} \preceq \mf{w'}\end{math}.  It suffices to check that the resulting diagram is reduced, and represents a PDS; the other needed properties follow from the previous discussion.

First, we show that the wiring diagram $\mf{\bar{u}'}$ obtained by concatenating $\mf{v'}$ and $\mf{x'}$ is reduced.  For this, it is enough to show that the factorization $u' = x'v'$ is length-additive, or equivalently, that \begin{math}x' \leq_{(r)} u'.\end{math}  

By an \textit{inversion} of a permutation $\sigma$, we mean a pair of values $a<b$ with $\sigma^{-1}(a) > \sigma^{-1}(b).$  By the usual criterion for comparison in the right weak order, we must show that every inversion of $x'$ is an inversion of $u'$.
This follows by construction for any inversion which does not involve the value $q$.  
The remaining inversions correspond to wires $b$ which cross wire $q$.  
Since \begin{math}(q \uparrow b)\end{math} for each such $b$, we have only pairs $b > q$ with \begin{math}x'^{-1}(b) < x'^{-1}(q).\end{math}  
We claim \begin{math}u'^{-1}(b) < u'^{-1}(q).\end{math}  If $b \in u'([k])$, this is obvious, since \begin{math}u'^{-1}(q) > k\end{math}; otherwise, $u'^{-1}(b)$, $u'^{-1}(q) \in [k+1,n]$, and the result follows from the fact that $u'$ is anti-Grassmannian.

Next, let $\mf{\bar{w}'}$ be the diagram obtained by adding the bridges inherited from the diagram \begin{math}\mf{u} \preceq \mf{w}\end{math} to the diagram $\mf{\bar{u}'}$.  We must show that $\mf{w}'$ is reduced.  We argue by induction on the number of bridges.  Suppose inserting crossings corresponding to the bridges 
\begin{displaymath} (a_1,b_1),\ldots,(a_{r-1},b_{r-1})\end{displaymath}
gives a reduced diagram $\mf{w^*}$, where the endpoints of the bridges have been renumbered to reflect the addition of wire $q$. 
Consider what happens when we add the bridge $(a_r,b_r)$.  It follows from Remark \ref{nocross}, and the fact that $\mf{u}$ is a PDS for $u$ in $\mf{w}$, that wires $a_r$ and $b_r$ do not cross again after the bridge $(a_r,b_r)$.   Hence, it suffices to show they do not cross before the bridge in the diagram for $\mf{w^*}$.  Since \begin{math}a_r,b_r \neq q,\end{math} it follows from our construction that the two wires cross before the bridge in the diagram for $\mf{w^*}$ if and only if the corresponding wires cross before the corresponding bridge in the diagram for $\mf{w}$.  Since $\mf{w}$ is reduced, these wires do not cross, and $\mf{w'}$ is reduced by induction.  

Hence, setting \begin{math}\mf{u'} = {\mf{\bar{u}'}}\end{math} and \begin{math}\mf{w'}={\mf{\bar{w}'}}\end{math} gives a reduced bridge diagram \begin{math}\mf{u'} \preceq \mf{w'}\end{math}. From the previous paragraph, and Remark \ref{nocross}, we see that $\mf{u'}$ is a PDS of $\mf{w'}$, and hence the diagram is valid.  
This completes the proof.  
\end{proof}

 \begin{figure}[ht]
 \centering
 \begin{subfigure}[t]{0.35 \textwidth}
 \centering
 \begin{tikzpicture}[scale = 0.6]
\draw(3,6) -- (3,0); 
\draw (0,1) -- (3,1); 
\draw (2,2) -- (3,2);
\draw (0,3) -- (3,3);
\draw (2,4) -- (3,4);
\draw (1,5) -- (3,5);
\edge{0}{1}{2};
\edge{1}{3}{2};
\edge{2}{3}{1};
\node [right] at (3,1) {$5$};\node [right] at (3,2) {$4$};
\node [right] at (3,3) {$3$};\node [right] at (3,4) {$2$};
\node [right] at (3,5) {$1$};
\bdot{3}{1};\wdot{3}{2};\wdot{3}{3};\bdot{3}{4};\bdot{3}{5};
\bdot{2}{2}; \wdot{1.5}{3};
 \end{tikzpicture}
 \caption{Bridge graph $G$ for\\
  $f = [5,6,7,4,8] \in \Bd(3,5)$.}
 \end{subfigure}
 \hfill
 \begin{subfigure}[t]{0.6 \textwidth}
 \centering
 \begin{tikzpicture}[scale = 0.6]
 \draw[white] (0,-1) -- (1,-1);
 \btm{0}{0};\btm{0}{1};\btm{0}{2};\crsg{0}{3};
 \btm{1}{0};\btm{1}{1};\crsg{1}{2};\tp{1}{3};
 \btm{2}{0};\brd{2}{1};\tp{2}{2};\tp{2}{3};
 \crsg{3}{0};\tp{3}{1};\tp{3}{2};\tp{3}{3};
 \btm{4}{0};\btm{4}{1};\brd{4}{2};\tp{4}{3};
 \btm{5}{0};\btm{5}{1};\btm{5}{2};\crsg{5}{3};
 \btm{6}{0};\btm{6}{1};\brd{6}{2};\tp{6}{3};
 \draw[dotted, very thick] (2,-0.5) -- (2,4.5);
  \draw[dotted, thick] (4,-0.5) -- (4,4.5);
   \draw[dotted, thick] (5,-0.5) -- (5,4.5);
 \node[left] at (0,0) {$5$}; \node[left] at (0,1) {$4$}; \node[left] at (0,2) {$3$};
 \node[left] at (0,3) {$2$}; \node[left] at (0,4) {$1$};
 \node[right] at (7,0) {$5$}; \node[right] at (7,1) {$4$}; \node[right] at (7,2) {$3$};
 \node[right] at (7,3) {$2$}; \node[right] at (7,4) {$1$};
 \end{tikzpicture}
 \captionsetup{singlelinecheck=off}
\caption[.]{$f$ corresponds to \begin{math}\langle u, w \rangle_3\end{math} where \begin{math}u = 32154,\end{math} \begin{math}w = 53214.\end{math}
 We have a bridge diagram for $G$, \begin{displaymath}\mf{w} = s_2 \bm{s_1}s_2\bm{s_4}s_3\bm{s_2}\bm{s_1}.\end{displaymath}  Letters in the PDS for $u$ in $\mf{w}$ are bolded.}
 \end{subfigure}
 \begin{subfigure}[t]{0.35 \textwidth}
 \centering
 \begin{tikzpicture}[scale = 0.6]
\draw(3,6) -- (3,-1); 
\draw (2,4) -- (3,4);
\draw (0,0) -- (3,0); 
\draw (0,2) -- (3,2);
\draw (2,3) -- (3,3);
\draw (1,5) -- (3,5);
\draw (2,1) -- (3,1);
\edge{1}{2}{3};
\edge{0}{0}{2};
\edge{2}{2}{1};
\bdot{2}{4}; \wdot{1.5}{2}; \bdot{2}{1};
\node [right] at (3,0) {$6$};\node [right] at (3,1) {$5$};
\node [right] at (3,2) {$4$};\node [right] at (3,3) {$3$};
\node [right] at (3,4) {$2$}; \node [right] at (3,5) {$1$}; 
\wdot{3}{0};\wdot{3}{1};\wdot{3}{2};\bdot{3}{3};\wdot{3}{4};\bdot{3}{5};
 \end{tikzpicture}
 \caption{Adding a black lollipop\\
 gives a bridge graph $G'$ for\\
 \begin{math} f' = [6,2,7,9,5,10] \in \Bd(3,6)\end{math}}
 \end{subfigure}
 \hfill
 \begin{subfigure}[t]{0.6 \textwidth}
 \centering
 \begin{tikzpicture}[scale =  0.6]
 \draw [white] (0,-1) -- (1,-1);
 \draw [dashed, rounded corners] (0,0) -- (2.2,0) -- (2.8,1) 
 -- (4.2,1) -- (4.8,2) -- (7.2,2) 
 -- (7.8,3) -- (9.2,3) -- (9.8,4)
 -- (11,4);
 \btm{0}{1};\btm{0}{2};\btm{0}{3};\crsg{0}{4};
  \btm{1}{1};\btm{1}{2};\crsg{1}{3};\tp{1}{4};
 \tcrsg{2}{0}; \tp{2}{1};\tp{2}{2};\tp{2}{3};\tp{2}{4};
 \btm{3}{0}; \brd{3}{2}; \tp{3}{3}; \tp{3}{4};
 \btm{4}{0};\tcrsg{4}{1};\tp{4}{2};\tp{4}{3};\tp{4}{4};
 \crsg{5}{0};\tp{5}{2};\tp{5}{3};\tp{5}{4};
 \btm{6}{0};\btm{6}{1};\brd{6}{3};\tp{6}{4};
 \btm{7}{0};\btm{7}{1};\tcrsg{7}{2};\tp{7}{3}; \tp{7}{4};
 \btm{8}{0};\btm{8}{1};\btm{8}{2};\crsg{8}{4};
 \btm{9}{0};\btm{9}{1};\btm{9}{2};\tcrsg{9}{3};\tp{9}{4};
 \btm{10}{0};\btm{10}{1};\brd{10}{2};\tp{10}{4};
 \node [left] at (0,0) {$6$}; \node [left] at (0,1) {$5$}; \node [left] at (0,2) {$4$}; \node [left] at (0,3) {$3$};
  \node [left] at (0,4) {$2$}; \node [left] at (0,5) {$1$};
   \node [right] at (11,0) {$6$}; \node [right] at (11,1) {$5$}; \node [right] at (11,2) {$4$}; \node [right] at (11,3) {$3$};
  \node [right] at (11,4) {$2$}; \node [right] at (11,5) {$1$};
\draw[dotted, very thick] (3,-0.5) -- (3,5.5);
\draw[dotted, thick] (6,-0.5) -- (6,5.5);
\draw[dotted, thick] (8,-0.5) -- (8,5.5);
 \end{tikzpicture}
\caption{$f'$ corresponds to $\langle u', w' \rangle_3$, where $u' = 431652$,\\
$w' = 643152$.  We build a bridge diagram for $G'$ by adding a wire (dashed) with right endpoint at position $2$.
The result corresponds to $\mf{w}' = s_3\bm{s_2}\bm{s_1}\bm{s_3}s_2\bm{s_5}\bm{s_4}s_3\bm{s_5}\bm{s_2}\bm{s_1}$.}
 \end{subfigure}
 \caption{Adding a lollipop to the bridge graph $G$ gives a bridge graph $G'$.  We construct a bridge diagram for $G'$ by adding a new wire (dashed) to a bridge diagram for $G$.  The portion of each bridge diagram to the right of the thick vertical line is the restricted part. The thin vertical lines divide the restricted part of each diagram into segments, as in the proof of Lemma \ref{addfp}.}
  \label{addlolly}
 \end{figure}
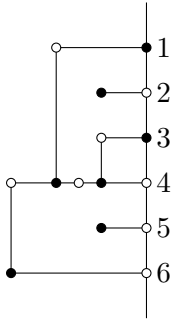
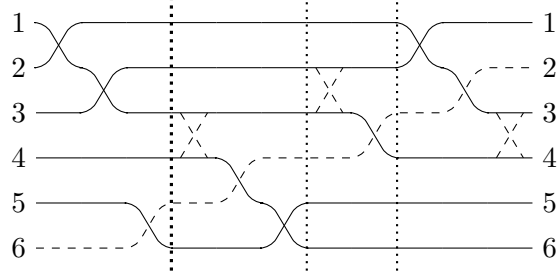
 
\begin{lem}
\label{nofp}
Let $G$ be a bridge graph which has no lollipops, corresponding to a positroid variety $\Pi_G$ in $\Gr(k,n)$. Then we can construct a valid bridge diagram for some $\mf{u} \preceq \mf{w}$ which corresponds to the bridge graph $G$.  
\end{lem}

\begin{proof}

We first establish the following claim.

\begin{claim}\label{red}  Let \begin{math}\mf{u'}\preceq \mf{w'}\end{math} be a valid bridge diagram corresponding to a bridge graph $G'$, where
\begin{equation} \Pi_{G'}=\Pi_{\langle u',w'\rangle_k}.\end{equation}  Suppose that adding an $(a,a+1)$ bridge to $G'$ yields a reduced graph $G$.  Then \begin{math}w'^{-1}(a) < w'^{-1}(a+1),\end{math} so adding an $(a,a+1)$ bridge adjacent to the right boundary of the bridge diagram \begin{math}\mf{u}' \preceq \mf{w}'\end{math} yields a valid bridge diagram for $G$.
\end{claim}

\begin{proof}
To prove the claim, let $f'$ be the bounded affine permutation corresponding to $G'$.  Then 
\begin{equation}f '= t_{u'([k])}u'w'^{-1}.\end{equation}
Since $G$ is a reduced bridge graph, we have \begin{math}f'(a) > f'(a+1).\end{math}  So one of the following holds: 
\begin{enumerate}
\item $w'^{-1}(a) \in [k]$ and $w'^{-1}(a+1) \in [k+1,n].$
\item $w'^{-1}(a), w^{-1}(a+1) \in [k]$ and $u'w'^{-1}(a) > u'w'^{-1}(a+1).$
\item $w'^{-1}(a),w'^{-1}(a+1)\in [k+1,n]$ and $u'w'^{-1}(a) > u'w'^{-1}(a+1).$
\end{enumerate}
In the first case, the fact that \begin{math}w'^{-1}(a) < w'^{-1}(a+1)\end{math} is obvious.  In the others, it follows from the fact that $u'$ is anti-Grassmannian.  So \begin{math}w'^{-1}(a) < w'^{-1}(a+1)\end{math}, and wires $a$ and $a+1$ never cross in the the diagram for $\mf{w'}$.  Hence adding the desired bridge gives a reduced diagram, and the claim follows.  
\end{proof}

Let $G$ be as above, and let \begin{math}d = \dim(\Pi_G).\end{math}  We proceed by induction on $d$.
There is only one lollipop-free bridge graph with $d=1$, corresponding to the big cell in $\mathbb{P}^2$.  The PDS \begin{math}\mf{1} \preceq \mf{s_1}\end{math} is a bridge diagram for $G$, and the base case is complete.

Say the result holds for every lollipop-free bridge graph with $d' < d$.  (Note that there are finitely many such bridge graphs.)  Let $G'$ be the graph obtained from $G$ by removing the last bridge.  Note that since $G$ contains no lollipops, the last bridge must be of the form $(a,a+1)$ for some $a$.  By Claim \ref{red}, it suffices to show that we can construct a valid bridge diagram for $G'$.
If neither $a$ nor $a+1$ is a lollipop in $G'$, then $G'$ does not have any lollipops, and this follows by inductive assumption.  

Next, suppose $a$ or $a+1$ is a lollipop in $G'$, and let $\bar{G}$ be the bridge graph obtained from $G'$ by deleting all lollipops.  By inductive assumption, we can construct a bridge diagram \begin{math}\bar{\mf{u}} \preceq \bar{\mf{w}}\end{math} corresponding to $\bar{G}$.  By lemma \ref{addfp}, we may successively add wires to the diagram as needed, to  produce a valid bridge diagram \begin{math}\mf{u}' \preceq \mf{w}'\end{math} for $G'$.  

In each case, Claim \ref{red} shows that we can insert an $(a,a+1)$ bridge on the far right side of the diagram \begin{math}\mf{u'} \preceq \mf{w'}\end{math}, and obtain the desired diagram for \begin{math}\mf{u} \preceq \mf{w}.\end{math}  This completes the proof.

\end{proof}

We can now prove the reverse direction of Theorem \ref{main}, which establishes the result.

\begin{prop}
Every parametrization arising from a bridge graph agrees with some projected Deodhar parametrization.
\end{prop}

\begin{proof}
Let $G$ be a bridge graph.  It is enough to show that we can build a valid bridge diagram corresponding to $G$.  Let $G'$ be the graph obtained by removing all lollipops from $G$.  Then we can build a valid bridge diagram for $G'$, by Lemma \ref{nofp}.  We then use Lemma \ref{addfp} to add lollipops as needed, and obtain a valid bridge diagram for $G$.\end{proof}

\section{Local moves for bridge diagrams}
\label{isotopy}

Let \begin{math}u \leq w \in S_n,\end{math} let $\mf{w}$ be a reduced word for $w$, and let \begin{math}\mf{u} \preceq \mf{w}\end{math} be the PDS for $u$ in $\mf{w}$. Performing a Coxeter move on $\mf{w}$ yields a new word $\mf{w}'$ for $w$.  Let $\mf{u}'$ be the unique PDS for $u$ in $\mf{w}'$.  In this section, we show how to construct the bridge diagram \begin{math}\mf{u}' \preceq \mf{w}'\end{math} by performing a local transformation on the diagram \begin{math}\mf{u} \preceq \mf{w}.\end{math}  We call these local transformations \emph{PDS moves}.  Rietsch exhibits complete sets of PDS moves for all finite Weyl groups in \citep{Rie08}, without using the language of bridge diagrams.  

Since any reduced word for $w$ can be transformed into any other by applying Coxeter moves, any bridge diagram for a PDS of $u \leq w$ can be transformed into any other using PDS moves.  To perform a PDS move on the diagram \begin{math}\mf{u} \preceq \mf{w},\end{math} we first perform the desired Coxeter move on the diagram for $\mf{w}$.  We then choose some of the affected  crossings to be bridges in our new diagram, as described below. 

 In the case of a commutation move \begin{math}s_is_j = s_js_i,\end{math} the generator $s_i$ (respectively, $s_j$) corresponds to a bridge in \begin{math}\mf{u}' \preceq \mf{w}'\end{math} if and only if the same is true in \begin{math}\mf{u} \preceq \mf{w}.\end{math} For braid moves, the situation depends on the configuration of bridges in $\mf{w}$. We summarize the situation in Table \ref{moves} below, illustrated in Figure \ref{legal}.  (The terms ``legal" and ``illegal" will be explained in the next section.)  In each case, the generators contained in the PDS are bolded, while bridges are in ordinary type.

\begin{table}[ht]
\centering
\caption{Braid moves for PDS's.  Factors in the PDS are bolded, bridges are non-bolded.}
\begin{tabular}{| l | l |}
\hline
\text{Legal Moves} & \text{ Illegal Moves}\\\hline
 \begin{math} s_{i+1}\bm{s_is_{i+1}} \leftrightarrow \bm{s_is_{i+1}}s_i\end{math} & \begin{math}s_{i+1}s_i\bm{s_{i+1}} \leftrightarrow s_i\bm{s_{i+1}}s_i\end{math}\\ 
 \begin{math} \bm{s_{i+1}s_{i}}s_{i+1} \leftrightarrow s_i\bm{s_{i+1}s_i}\end{math} & \begin{math} s_{i+1}\bm{s_i}s_{i+1} \leftrightarrow s_is_{i+1}\bm{s_i}\end{math}\\
 \begin{math} \bm{s_{i+1}s_is_{i+1}}  \leftrightarrow \bm{s_is_{i+1}s_i}\end{math} & \begin{math}s_{i+1}s_is_{i+1} \leftrightarrow s_is_{i+1}s_i\end{math}\\  \hline
\end{tabular}
\label{moves}
\end{table}

\begin{prop}
Let \begin{math}\mf{u} \preceq \mf{w},\end{math} where $\mf{u}$ is a PDS for $\mf{w}$.  Let $\mf{w}'$ be obtained from $\mf{w}$ by performing a braid move, and let \begin{math}\mf{u}' \preceq \mf{w}'\end{math} be the PDS for $u$ in $\mf{w}'$. Then performing the corresponding local move from Table \ref{moves} on the bridge diagram for \begin{math}\mf{u} \preceq \mf{w}\end{math} yields the bridge diagram for \begin{math}\mf{u}' \preceq\mf{w}'.\end{math} 
\end{prop}
\begin{proof}
 
We sketch a proof using bridge diagrams.  First, note that the table covers all possible cases.  We cannot, for example, have the configuration \begin{math}\bm{s_i}s_{i+1}\bm{s_i}\end{math} since a PDS must be reduced.

We now verify that each move yields the desired bridge diagram. In each case, the bolded generators give the same permutation before and after the move.  So the bolded generators, together with the generators from  $\mf{u}$ in the rest of the diagram, give a subexpression for $u$ in $\mf{w'}$.  We must show that this new subexpression is a PDS.  In the language of bridge diagrams, this is equivalent to saying that each bridge is inserted between two wires in the underlying diagram for $\mf{u}$ which never cross again.

This is clear by inspection, except perhaps for the two moves 
\begin{equation} s_{i+1}s_i\bm{s_{i+1}} = s_i\bm{s_{i+1}}s_i\end{equation}
\begin{equation}s_{i+1}\bm{s_i}s_{i+1}=s_is_{i+1}\bm{s_i},\end{equation}
illustrated in Figure \ref{l1d} and Figure \ref{l1e}.  We will show that the claim holds for the former case; the argument in the later case is similar.

Consider the pair of diagrams in Figure \ref{l1d}.  Suppose the diagram at left represents a piece of a bridge diagram $\mf{u} \preceq \mf{w}$.  Since the subexpression $\mf{u} \preceq \mf{w}$ is a PDS, wires $a_1$ and $a_2$ never cross after the location of the braid move, and similarly for $a_2$ and $a_3$.  But then $a_1$ and $a_3$ cannot cross after the location of the braid move.  Hence if we replace the portion of $\mf{u} \preceq \mf{w}$ corresponding to the diagram on the left side of Figure \ref{l1d} with the diagram on the right, then each bridge is inserted between pairs of wires which never cross again, and we have a PDS.  The proof in the other direction is clear.
\end{proof}

\begin{figure}[ht]
 \begin{subfigure}[b]{\textwidth}
 \renewcommand{\thesubfigure}{\arabic{subfigure}}
 \centering
 \begin{subfigure}[b]{0.4 \textwidth}
  \renewcommand{\thesubfigure}{\alph{subfigure}}
 \centering
 \begin{tikzpicture}[scale = 0.6]
 \crsg{0}{0};\tp{0}{1};
 \btm{1}{0};\crsg{1}{1};
 \brd{2}{0};\tp{2}{1};
  \node[right] at (3,0) {$a_3$}; 
 \node[right] at (3,1) {$a_2$}; 
 \node[right] at (3,2) {$a_1$};
\begin{scope}[xshift = 4.5 cm]
 \btm{0}{0};\brd{0}{1};
 \crsg{1}{0};\tp{1}{1};
 \btm{2}{0};\crsg{2}{1};
 \node[right] at (3,0) {$a_3$}; 
 \node[right] at (3,1) {$a_2$}; 
 \node[right] at (3,2) {$a_1$};
 \end{scope}
 \end{tikzpicture}
 \caption{$s_{i+1}\bm{s_is_{i+1}} \leftrightarrow \bm{s_is_{i+1}}s_i$}
 \end{subfigure}
  \begin{subfigure}[b]{0.4 \textwidth}
 \renewcommand{\thesubfigure}{\alph{subfigure}}
  \centering
 \begin{tikzpicture}[scale = 0.6]
 \brd{0}{0};\tp{0}{1};
 \btm{1}{0};\crsg{1}{1};
 \crsg{2}{0};\tp{2}{1};
  \node[right] at (3,0) {$a_3$}; 
 \node[right] at (3,1) {$a_2$}; 
 \node[right] at (3,2) {$a_1$};
\begin{scope}[xshift = 4.5 cm]
 \btm{0}{0};\crsg{0}{1};
 \crsg{1}{0};\tp{1}{1};
 \btm{2}{0};\brd{2}{1};
  \node[right] at (3,0) {$a_3$}; 
 \node[right] at (3,1) {$a_2$}; 
 \node[right] at (3,2) {$a_1$};
 \end{scope}
 \end{tikzpicture}
 \caption{$\bm{s_{i+1}s_i}s_{i+1} \leftrightarrow s_i\bm{s_{i+1}s_i}$}
 \end{subfigure}
 
 \vspace{0.5 cm}
  \begin{subfigure}[b]{0.4 \textwidth}
 \renewcommand{\thesubfigure}{\alph{subfigure}}
  \centering
 \begin{tikzpicture}[scale = 0.6]
 \crsg{0}{0};\tp{0}{1};
 \btm{1}{0};\crsg{1}{1};
 \crsg{2}{0};\tp{2}{1};
  \node[right] at (3,0) {$a_3$}; 
 \node[right] at (3,1) {$a_2$}; 
 \node[right] at (3,2) {$a_1$};
\begin{scope}[xshift = 4.5 cm]
 \btm{0}{0};\crsg{0}{1};
 \crsg{1}{0};\tp{1}{1};
 \btm{2}{0};\crsg{2}{1};
  \node[right] at (3,0) {$a_3$}; 
 \node[right] at (3,1) {$a_2$}; 
 \node[right] at (3,2) {$a_1$};
 \end{scope}
 \end{tikzpicture}
 \caption{$\bm{s_{i+1}s_is_{i+1}} \leftrightarrow \bm{s_is_{i+1}s_i}$}
 \end{subfigure}
  \addtocounter{subfigure}{-3}
 \caption{Legal braid moves for bridge diagrams.  These moves preserve isotopy class.}
 \renewcommand{\thesubfigure}{\arabic{subfigure}}
 \end{subfigure}
 
 \begin{subfigure}[b]{\textwidth}
 \renewcommand{\thesubfigure}{\arabic{subfigure}}
 \centering
 \begin{subfigure}[b]{0.4 \textwidth}
   \addtocounter{subfigure}{3}
  \renewcommand{\thesubfigure}{\alph{subfigure}}
 \centering
 \begin{tikzpicture}[scale = 0.6]
\crsg{0}{0};\tp{0}{1};
 \btm{1}{0};\brd{1}{1};
 \brd{2}{0};\tp{2}{1};
  \node[right] at (3,0) {$a_3$}; 
 \node[right] at (3,1) {$a_2$}; 
 \node[right] at (3,2) {$a_1$};
\begin{scope}[xshift = 4.5 cm]
 \btm{0}{0};\brd{0}{1};
 \crsg{1}{0};\tp{1}{1};
 \btm{2}{0};\brd{2}{1};
  \node[right] at (3,0) {$a_3$}; 
 \node[right] at (3,1) {$a_2$}; 
 \node[right] at (3,2) {$a_1$};
 \end{scope}
 \end{tikzpicture}
 \addtocounter{subfigure}{-1}
 \caption{\begin{math}s_{i+1}s_i\bm{s_{i+1}} \leftrightarrow s_i\bm{s_{i+1}}s_i\end{math}}
 \label{l1d}
 \end{subfigure}
 \begin{subfigure}[b]{0.4 \textwidth}
   \renewcommand{\thesubfigure}{\alph{subfigure}}
 \centering
 \begin{tikzpicture}[scale = 0.6]
 \brd{0}{0};\tp{0}{1};
 \btm{1}{0};\crsg{1}{1};
 \brd{2}{0};\tp{2}{1};
  \node[right] at (3,0) {$a_3$}; 
 \node[right] at (3,1) {$a_2$}; 
 \node[right] at (3,2) {$a_1$};
\begin{scope}[xshift = 4.5 cm]
 \btm{0}{0};\crsg{0}{1};
 \brd{1}{0};\tp{1}{1};
 \btm{2}{0};\brd{2}{1};
  \node[right] at (3,0) {$a_3$}; 
 \node[right] at (3,1) {$a_2$}; 
 \node[right] at (3,2) {$a_1$};
 \end{scope}
 \end{tikzpicture}
 \caption{\begin{math}s_{i+1}\bm{s_i}s_{i+1} \leftrightarrow s_is_{i+1}\bm{s_i}\end{math}}
  \label{l1e}
 \end{subfigure}
  \vspace{0.5 cm}
 
  \begin{subfigure}[b]{0.4 \textwidth}
    \renewcommand{\thesubfigure}{\alph{subfigure}}
  \centering
 \begin{tikzpicture}[scale = 0.6]
 \brd{0}{0};\tp{0}{1};
 \btm{1}{0};\brd{1}{1};
 \brd{2}{0};\tp{2}{1};
  \node[right] at (3,0) {$a_3$}; 
 \node[right] at (3,1) {$a_2$}; 
 \node[right] at (3,2) {$a_3$};
\begin{scope}[xshift = 4.5 cm]
 \btm{0}{0};\brd{0}{1};
 \brd{1}{0};\tp{1}{1};
 \btm{2}{0};\brd{2}{1};
  \node[right] at (3,0) {$a_3$}; 
 \node[right] at (3,1) {$a_2$}; 
 \node[right] at (3,2) {$a_3$};
 \end{scope}
 \end{tikzpicture}
 \caption{\begin{math}s_{i+1}s_is_{i+1} \leftrightarrow s_is_{i+1}s_i\end{math}}
 \end{subfigure}
  \addtocounter{subfigure}{-5}
 \caption{Illegal braid moves for bridge diagrams.  These moves change isotopy class.}
  \end{subfigure}
  \caption{Braid moves for PDS's}
 \label{legal}
 \end{figure}

\subsection{Isotopy classes of bridge diagrams}

Let $\mf{w}$ and $\mf{w'}$ be reduced words for $w$, and let \begin{math}\mf{u} \preceq \mf{w}\end{math} and \begin{math}\mf{u'} \preceq \mf{w'}\end{math} be the PDS's for $u$ in \begin{math}\mf{w}$ and $\mf{w'}\end{math} respectively.  We say the bridge diagrams \begin{math}\mf{u} \preceq \mf{w}\end{math} and \begin{math}\mf{u'} \preceq \mf{w'}\end{math} are \emph{isotopic} if they have the same sequence of bridges 
\begin{displaymath}(a_1,b_1),\ldots,(a_d,b_d).\end{displaymath}  
If $u$ is anti-Grassmannian, the valid bridge diagrams \begin{math}\mf{u} \preceq \mf{w}\end{math} and \begin{math}\mf{u'} \preceq \mf{w'}\end{math} correspond to isotopic bridge graphs.  

We call a PDS move \emph{legal} if it preserves isotopy class, and illegal otherwise.  Any commutation move is legal.  From Figure \ref{legal}, we see that a braid move is legal if and only if it involves at most one bridge; or equivalently, if and only if it involves as least two factors of the PDS.  We say two bridge diagrams for PDS's are \emph{move-equivalent} if we can transform one into the other by a sequence of legal moves.  We will show that any two valid isotopic bridge diagrams are move equivalent. Thus the legal moves define equivalence classes of valid bridge diagrams, and these equivalence classes are in bijection with bridge graphs. 

For \begin{math}x,y \in S_n,\end{math} and $\mf{y}$ a reduced word for $y$, let $\mf{x} \preceq \mf{y}$ be a bridge diagram representing the PDS for $x$ in $\mf{y}$.  Let $(a,b)$ be a bridge in $B$, and let \begin{math}a < c < b\end{math}.  Then we have either \begin{math}(a,b) \rightarrow (c \downarrow a)\end{math} or \begin{math}(a,b) \rightarrow (c \uparrow b).\end{math} We say $B$ is \emph{$k$-divided} if, for all such $a,b$ and $c$, we have 
\begin{equation}
\begin{cases} (c \downarrow a) & \text{if }x^{-1}(c) \leq k\\
(c \uparrow b) & \text{if }x^{-1}(c) > k
\end{cases}
\end{equation}

\begin{rmk} By Lemma \ref{escape}, any valid bridge diagram is $k$-divided.  \end{rmk}
 
\begin{prop}
\label{divided}
Let $u \leq w$.  Let \begin{math}B_1=\mf{u} \preceq \mf{w}\end{math} and \begin{math}B_2=\mf{u'} \preceq \mf{w'}\end{math} denote isotopic, $k$-divided bridge diagrams corresponding to PDS's for $u$ in reduced words of $w$.  Then $B_1$ and $B_2$ are move-equivalent. 
\end{prop}

\begin{proof}

We induce on the number of bridges in the diagrams.  For the base case, note that any move involving at most one bridge is legal.  Hence, any two isotopic bridge diagrams with at most one bridge are move equivalent.  Assume the claim holds for diagrams with at most $d-1$ bridges, and suppose $B_1$ and $B_2$ each have $d$ bridges.  

Suppose the rightmost bridge of each $B_i$ is a bridge $(a,b)$, where $a < b$. For $i=1,2$, let $B_i^0$ be the portion of $B_i$ which extends from the rightmost bridge to the right edge of the diagram.  
Then we have
\begin{equation} B_1^0 = \bar{\mf{u}} \preceq \bar{\mf{w}}\end{equation}
\begin{equation} B_2^0 = {\mf{\bar{u}'}} \preceq {\mf{\bar{w}'}} \end{equation}
for some \begin{math}\bar{u},\bar{u}',\bar{w},\bar{w}' \in S_n,\end{math} and each $B_i^0$ is a bridge diagram for a PDS.
Next, we construct a reduced bridge diagram $B^0$ with a single bridge $(a,b)$ adjacent to its left edge, which satisfies all of the following for each \begin{math}c \in [n]\end{math}:
\begin{enumerate}
\item If $c < a $ or $c > b$, then wire $c$ runs straight from the left endpoint labeled $c$ to the right, and does not cross any other wires.
\item If \begin{math}a < c < b,\end{math} we have 
\begin{equation}
\begin{cases} (a,b) \rightarrow (c \downarrow a) & u^{-1}(c) \leq k\\
(a,b) \rightarrow (c \uparrow b) & u^{-1}(c) > k
\end{cases}
\end{equation}
\item If $a < c,c' < b$ and either we have
\begin{displaymath} u^{-1}(c),u^{-1}(c') \leq k\end{displaymath} 
or we have
\begin{displaymath} u^{-1}(c),u^{-1}(c') > k\end{displaymath}
then wires $c$ and $c'$ do not cross.
\end{enumerate}
The fact that such a diagram exists follows easily from the proof of Lemma \ref{addfp}, and we have 
\begin{equation} B^0 = \mf{u}_* \preceq \mf{w}_*\end{equation} 
for some \begin{math}u_*,w_* \in S_n.\end{math}   See Figure \ref{kdiv}.

Note that every inversion of $u_*$ is also an inversion of $\bar{u}$ and $\bar{u}'$, since $B_1$ and $B_2$ are $k$-divided.  By the usual criterion, it follows that \begin{math}u_* \leq_{(r)} \bar{u},\bar{u}'\end{math}.  Hence we can add additional crossings to the diagram $\mf{u_*}$ on the left to build a reduced diagram $\bar{\mf{u}}_*$ for $\bar{u}$, and to build a reduced diagram $\mf{\bar{u}'_*}$ for $\bar{u}'$.

Since the reduced diagrams $B_1^0$ and $B_2^0$ each have only one bridge $(a,b)$, it follows that wires $a$ and $b$ do not cross in the wiring diagrams $\bar{\mf{u}}$ and $\mf{\bar{u}'}$, and thus do not cross in the diagrams $\bar{\mf{u}}_*$ and $\mf{\bar{u}'_*}$. Hence we can add an $(a,b)$ bridge to each of $\bar{\mf{u}}_*$ and $\mf{\bar{u}'_*}$, with the bridge just to the left left of the copy of $\mf{u_*}$ in each diagram.  The result is a pair of reduced diagrams $C_1^0$ and $C_2^0$ for PDS's, which are isotopic to $B_1^0$ and $B_2^0$ respectively, and whose rightmost entries form a copy of $B^0$.

By the base case, $C_i^0$ is move-equivalent to $B_i^0$ for $i = 1,2$.  So we can transform $B_1$, $B_2$ into bridge diagrams $B_1'$, $B_2'$, whose rightmost entries form a copy $B^0$.  Let $B_i^*$ be the part of $B_i'$ to the left of the copy of $B^0$, for $i = 1,2$.  Then $B_1^*$ and $B_2^*$ are isotopic, $k$-divided bridge diagrams with $d-1$ bridges, and are hence move-equivalent by induction.  Thus $B_1'$ and $B_2'$ are move equivalent, and hence so are $B_1$ and $B_2$.  
\end{proof}

\begin{cor}
Any two isotopic valid bridge diagrams are move equivalent.  
\end{cor}

 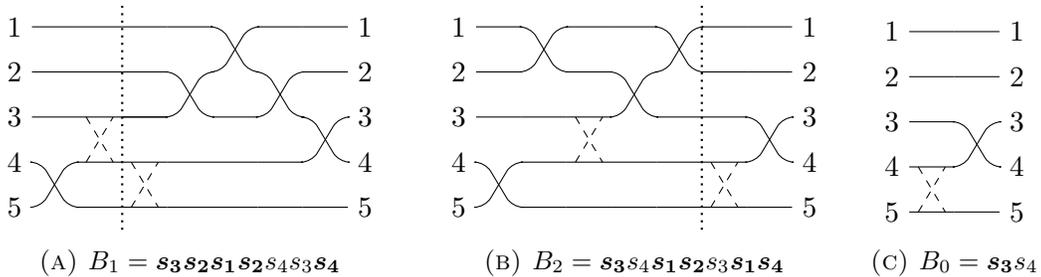
\begin{figure}[ht]
 \centering
 \begin{subfigure}[b]{0.35 \textwidth}
 \centering
 \begin{tikzpicture}[scale = 0.6]
 \crsg{1}{0};\tp{1}{1};\tp{1}{2};\tp{1}{3};
 \btm{2}{0};\brd{2}{1};\tp{2}{2};\tp{2}{3};
 \brd{3}{0};\tp{3}{1};\tp{3}{1};\tp{3}{2};\tp{3}{3};
\btm{4}{0};\btm{4}{1};\crsg{4}{2};\tp{4}{3};
\btm{5}{0};\btm{5}{1};\btm{5}{2};\crsg{5}{3};
\btm{6}{0};\btm{6}{1};\crsg{6}{2};\tp{6}{3};
\btm{7}{0};\crsg{7}{1};\tp{7}{2};\tp{7}{3};
\node[left] at (1,4) {1};\node[left] at (1,3) {2};\node[left] at (1,2) {3};\node[left] at (1,1) {4};\node[left] at (1,0) {5};
\node[right] at (8,4) {1};\node[right] at (8,3) {2};\node[right]  at (8,2) {3};\node[right]  at (8,1) {4};\node[right] at (8,0) {5};
\draw[dotted, thick] (3,-0.5) -- (3,4.5);
 \end{tikzpicture}
\caption{\begin{math}B_1 = \bm{s_3s_2s_1s_2}s_4s_3\bm{s_4}\end{math}}
 \end{subfigure}
 \begin{subfigure}[b]{0.35 \textwidth}
\centering
 \begin{tikzpicture}[scale = 0.6]
 \crsg{1}{0};\tp{1}{1};\tp{1}{2};\tp{1}{3};
 \btm{2}{0};\btm{2}{1};\btm{2}{2};\crsg{2}{3};
 \btm{3}{0};\brd{3}{1};\tp{3}{2};\tp{3}{3};
\btm{4}{0};\btm{4}{1};\crsg{4}{2};\tp{4}{3};
\btm{5}{0};\btm{5}{1};\btm{5}{2};\crsg{5}{3};
\brd{6}{0};\tp{6}{1};\tp{6}{2};\tp{6}{3};
\btm{7}{0};\crsg{7}{1};\tp{7}{2};\tp{7}{3};
\draw[dotted, thick] (6,-0.5) -- (6,4.5);
\node[left] at (1,4) {1};\node[left] at (1,3) {2};\node[left] at (1,2) {3};\node[left] at (1,1) {4};\node[left] at (1,0) {5};
\node[right] at (8,4) {1};\node[right] at (8,3) {2};\node[right]  at (8,2) {3};\node[right]  at (8,1) {4};\node[right] at (8,0) {5};
 \end{tikzpicture}
\caption{\begin{math}B_2 = \bm{s_3}s_4\bm{s_1s_2}s_3\bm{s_1s_4}\end{math}}
 \end{subfigure}
 \begin{subfigure}[b]{0.15 \textwidth}
\centering
 \begin{tikzpicture}[scale = 0.6]
 \brd{1}{0};\tp{1}{1};\tp{1}{2};\tp{1}{3};
 \btm{2}{0};\crsg{2}{1};\tp{2}{2};\tp{2}{3};
\node[left] at (1,4) {1};\node[left] at (1,3) {2};\node[left] at (1,2) {3};\node[left] at (1,1) {4};\node[left] at (1,0) {5};
\node[right] at (3,4) {1};\node[right] at (3,3) {2};\node[right]  at (3,2) {3};\node[right]  at (3,1) {4};\node[right] at (3,0) {5};
 \end{tikzpicture}
  \caption{$B_0 = \bm{s_3}s_4$}
 \end{subfigure}
 \caption{$B_1$ and $B_2$ are isotopic valid bridge diagrams with $u=42153$ and $w = 42531$.  For $i=1,2$, the portion of $B_i$ to the right of the dashed line is $B_i^0$.}
 \label{kdiv}
 \end{figure}

\clearpage

\bibliographystyle{abbrvnat}
\bibliography{combo}

\begin{thebibliography}{19}
\providecommand{\natexlab}[1]{#1}
\providecommand{\url}[1]{\texttt{#1}}
\expandafter\ifx\csname urlstyle\endcsname\relax
  \providecommand{\doi}[1]{doi: #1}\else
  \providecommand{\doi}{doi: \begingroup \urlstyle{rm}\Url}\fi

\bibitem[Arkani-Hamed et~al.(2014)Arkani-Hamed, Bourjaily, Cachazo, Goncharov,
  Postnikov, and Trnka]{ABCGPT14}
N.~Arkani-Hamed, J.~L. Bourjaily, F.~Cachazo, A.~Goncharov, A.~Postnikov, and
  J.~Trnka.
\newblock Scattering amplitudes and the positive {G}rassmannian.
\newblock \emph{preprint}, 2014.
\newblock arXiv:1212.5605v2 [hep-th].

\bibitem[Bergeron and Sottile(1998)]{BS98}
N.~Bergeron and F.~Sottile.
\newblock Schubert polynomials, the {B}ruhat order, and the geometry of flag
  manifolds.
\newblock \emph{Duke Math. J.}, 95:\penalty0 373--423, 1998.
\newblock arXiv:alg-geom/9703001.

\bibitem[Bjorner and Brenti(2005)]{BB05}
A.~Bjorner and F.~Brenti.
\newblock \emph{Combinatorics of {C}oxeter Groups}.
\newblock Springer Science+Business Media, Inc., New York, USA, 2005.

\bibitem[Brown et~al.(2006)Brown, Goodearl, and Yakimov]{BGY06}
K.~A. Brown, K.~R. Goodearl, and M.~Yakimov.
\newblock Poisson structures on affine spaces and flag varieties. {I}. {M}atrix
  affine {P}oisson space.
\newblock \emph{Advances in Mathematics}, 206\penalty0 (2):\penalty0 567--629,
  2006.
\newblock arXive:math/0509075v2 [math.QA].

\bibitem[Deodhar(1985)]{Deo85}
V.~V. Deodhar.
\newblock On some geometric aspects of {B}ruhat orderings. {I}. {A} finer
  decomposition of {B}ruhat cells.
\newblock \emph{Inventiones mathematicae}, 79:\penalty0 499--511, 1985.

\bibitem[Knutson et~al.(2013)Knutson, Lam, and Speyer]{KLS13}
A.~Knutson, T.~Lam, and D.~Speyer.
\newblock Positroid varieties: juggling and geometry.
\newblock \emph{Compositio Mathematica}, 149\penalty0 (10):\penalty0
  1710--1752, 2013.
\newblock arXiv:0903.3694v1 [math.AG].

\bibitem[Lam(2013{\natexlab{a}})]{Lam13B}
T.~Lam.
\newblock Notes on the totally nonnegative {G}rassmannian, 2013{\natexlab{a}}.
\newblock URL
  \url{http://www.math.lsa.umich.edu/~tfylam/Math665a/positroidnotes.pdf}.
\newblock Accessed: 09-26-2014.

\bibitem[Lam(2013{\natexlab{b}})]{Lam13a}
T.~Lam.
\newblock private communication, 2013{\natexlab{b}}.

\bibitem[Lusztig(1994)]{Lus94}
G.~Lusztig.
\newblock Total positivity in reductive groups.
\newblock In \emph{Lie theory and geometry}, volume 123 of \emph{Progr. Math.},
  pages 531--568. Birkh\"auser Boston, Boston, MA, 1994.

\bibitem[Lusztig(1998)]{Lus98}
G.~Lusztig.
\newblock Total positivity in partial flag varieties.
\newblock \emph{Representation Theory}, 2:\penalty0 70--78, 1998.

\bibitem[Marsh and Rietsch(2004)]{MR04}
R.~J. Marsh and K.~Rietsch.
\newblock Parametrizations of flag varieties.
\newblock \emph{Representation Theory}, 8:\penalty0 212--242 (electronic),
  2004.
\newblock arXiv:math/0307017.

\bibitem[Muller and Speyer(2014)]{MS14}
G.~Muller and D.~Speyer.
\newblock The twist for positroid varieties, 2014.
\newblock Unpublished manuscript.

\bibitem[Postnikov(2006)]{Pos06}
A.~Postnikov.
\newblock Total positivity, {G}rassmannians and networks.
\newblock \emph{preprint}, 2006.
\newblock arXiv:math/0609764 [math.CO].

\bibitem[Postnikov et~al.(2009)Postnikov, Speyer, and Williams]{PSW09}
A.~Postnikov, D.~E. Speyer, and L.~Williams.
\newblock Matching polytopes, toric geometry, and the totally non-negative
  {G}rassmannian.
\newblock \emph{J. Algebraic Combin.}, 30\penalty0 (2):\penalty0 173--191,
  2009.
\newblock arXiv:math/0706.2501v3.

\bibitem[Rietsch(1999)]{Rie99}
K.~Rietsch.
\newblock An algebraic cell decomposition on the nonnegative part of a flag
  variety.
\newblock \emph{Journal of algebra}, 213:\penalty0 144--154, 1999.
\newblock arXiv:alg-geom/9709035.

\bibitem[Rietsch(2006)]{Rie06}
K.~Rietsch.
\newblock Closure relations for totally nonnegative cells in {G}/{P}.
\newblock \emph{Mathematical research letters}, 13:\penalty0 775--786, 2006.
\newblock arXiv:math/0509137v2 [math.AG].

\bibitem[Rietsch(2008)]{Rie08}
K.~Rietsch.
\newblock A mirror symmetric construction of {$qH^\ast_T(G/P)_{(q)}$}.
\newblock \emph{Advances in Mathematics.}, 217\penalty0 (6):\penalty0
  2401--2442, 2008.
\newblock arXiv:math/0511124v2 [math.AG].

\bibitem[Talaska and Williams(2013)]{TW13}
K.~Talaska and L.~Williams.
\newblock Network parameterizations for the {G}rassmannian.
\newblock In \emph{25th {I}nternational {C}onference on {F}ormal {P}ower
  {S}eries and {A}lgebraic {C}ombinatorics ({FPSAC} 2013)}, Discrete Math.
  Theor. Comput. Sci. Proc., AS, pages 61--72. Assoc. Discrete Math. Theor.
  Comput. Sci., Nancy, 2013.
\newblock arXiv:1210.5433v2 [math.CO].

\bibitem[Williams(2007)]{Wil07}
L.~Williams.
\newblock Shelling totally nonnegative flag varieties.
\newblock \emph{J. Reine Angew. Math.}, 609:\penalty0 1--21, 2007.
\newblock arXive:math/0509129v1 [Math.RT].

\end{thebibliography}
\label{sec:biblio}

\end{document}